\documentclass[smallextended]{svjour3}    
\smartqed
\usepackage{graphicx}
\usepackage{amsfonts}
\usepackage{algorithm}
\usepackage{url}
\usepackage{algorithmic,amsmath}
\usepackage[caption=false]{subfig}
\usepackage{enumitem}

\newtheorem{thm}{Theorem}

\newtheorem{asm}{Assumption}

\allowdisplaybreaks[1]

\DeclareMathOperator*{\argmin}{arg\,min}
\numberwithin{theorem}{section}

\begin{document}

\title{A Sigmoidal Approximation for Chance-Constrained Nonlinear Programs\thanks {This work was partially supported by the U.S. Department of Energy grant DE-SC0014114.}}

\author{Yankai Cao       \and
        Victor M. Zavala  }
\institute{Yankai Cao \at
              Department of Chemical and Biological Engineering, University of British Columbia, 2360 East Mall, Vancouver, BC, V6T1Z3, Canada. \\
              \email{yankai.cao@ubc.ca}           \\
           \and
           Victor M. Zavala \at
           Department of Chemical and Biological Engineering, University of Wisconsin-Madison, 1415 Engineering Dr, Madison, WI 53706, USA.\\
           \email{victor.zavala@wisc.edu}
}


\date{Received: date / Accepted: date}

\maketitle

\begin{abstract}
We propose a sigmoidal approximation for the value-at-risk (that we call SigVaR)  and we use this approximation to tackle nonlinear programs (NLPs) with chance constraints. We prove that the approximation is conservative and that the level of conservatism can be made arbitrarily small for limiting parameter values. The SigVar approximation brings scalability benefits over exact mixed-integer reformulations because its sample average approximation can be cast as a standard NLP.  We also establish explicit connections between SigVaR and other smooth sigmoidal approximations recently reported in the literature.  We show that a key benefit of SigVaR over such approximations is that one can establish an explicit connection with the conditional value at risk (CVaR) approximation and exploit this connection to obtain initial guesses for the approximation parameters.  We present small- and large-scale numerical studies to illustrate the developments.  
\end{abstract}

\section{Problem Definition and Setting}
We study the chance-constrained nonlinear program (CC-P):
\begin{subequations}
\label{ccprob}
\begin{align}
\min \limits_{x \in  \mathcal{X} } \;\;  &\varphi(x)  \label{ccprob1} \\
{\rm s.t.}  \;\; &\mathbb{P}\left(f(x,\Xi)\leq 0\right) \geq 1- \alpha. \label{ccprob2}
\end{align}
\end{subequations}
Here, $x \in \mathbb{R}^n$ are decision variables and the objective function $\varphi: \mathbb{R}^n {\rightarrow} \mathbb{R}$ is twice continuously differentiable and potentially nonconvex. The set $\mathcal{X}:=\{\,x\,|g(x)\geq 0\}$ is assumed to be compact and non-empty and is comprised of twice differentiable and potentially nonconvex  constraints $g:\mathbb{R}^n\to \mathbb{R}^m$. We consider the probability space $(\Omega,\mathcal{F},\mathbb{P})$ and we assume that $\Omega$ is a measurable space equipped with $\sigma$-algebra $\mathcal{F}$ of subsets of $\Omega$, and that $\bar{\Xi}$ is a linear space of $\mathcal{F}$-measurable functions  $\Xi:\Omega \to \mathbb{R}^d$ (random variables). The probability measure function is given by $\mathbb{P}:\mathcal{F}\to [0,1]$ and we use $\xi\in \mathbb{R}^d$ to denote realizations of $\Xi$. The scalar constraint function $f: \mathbb{R}^n \times\bar{\Xi}{\rightarrow} \mathbb{R}$ is also assumed to be twice continuously differentiable and potentially nonconvex. We define the scalar random variable $Z:=f(x,\Xi)$ with realizations $z\in\mathbb{R}$. When appropriate, we use the notation $Z(x)$ to highlight the dependence of the random variable $Z$ on the decision $x$. We use $\mathbb{P}(Z \in D)$ to denote the probability of the event $Z\in D$ and $F_Z(t) =\mathbb{P}(Z \leq t)$ to denote the cumulative distribution functions of $Z$.

The CC \eqref{ccprob2} requires that the event $\{f(x,\Xi)\in (-\infty,0]\}$ occurs with probability of at least  $1- \alpha$, where $\alpha \in (0,1]$.  Since $\mathbb{P}\left(Z\leq 0\right)=F_{Z}(0)$, the CC can also be written as $F_{f(x,\Xi)}(0) \geq 1- \alpha$ or $1-F_{f(x,\Xi)}(0)\leq \alpha$. We recall that the $(1-\alpha)$-quantile of $Z$ is $Q_Z(1-\alpha)=\text{VaR}_{1-\alpha}(Z):=\min\{t\in \mathbb{R} : F_{Z}(t)\geq 1-\alpha\}$ (where VaR is known as the value-at-risk). Consequently, the CC can also be written as  $\text{VaR}_{1-\alpha}(f(x,\Xi))\leq 0$. Another important observation is that $\mathbb{E}[1_D(Z)]=\mathbb{P}(Z\in D)$ holds, where $1_D:\mathbb{R}\to \{0,1\}$ denotes the indicator function of set $D$ (i.e., $1_D(Z)=1$ if $Z \in D$ and $1_D(Z)=0$ if $Z \notin D$). Consequently, \eqref{ccprob2} can be written as $\mathbb{E}[1_{(-\infty,0]}(f(x,\Xi))]\geq 1-\alpha$ or, equivalently, as $\mathbb{E}[1_{(0,\infty)}(f(x,\Xi))]\leq \alpha$. We define the feasible set of CC-P as $\mathcal{X}(\alpha):=\mathcal{X}\cap \mathcal{P}(\alpha)$, where $\mathcal{P}(\alpha):=\{\,x\,|\mathbb{P}\left(f(x,\Xi)\leq 0\right) \geq 1- \alpha\}$ and we assume $\mathcal{X}(\alpha)$ to be compact for all $\alpha\in (0,1]$. We denote an optimal solution and objective value of \eqref{ccprob} as $x^*(\alpha)$ and $\varphi^*(\alpha)$, respectively. We also denote the solution set as $S^*(\alpha)$. We focus our attention on NLPs with a single CC but the concepts discussed can also be applied to multiple CCs of the form $\mathbb{P}\left(f_i(x,\Xi)\leq 0\right) \geq 1- \alpha_i,\, i=1,...,r$.  

A distinguishing and challenging feature of CC-P is that it cannot be solved exactly (except for certain simplified settings). Settings that admit exact solutions include those in which the quantile $Q_{f(x,\Xi)}(1-\alpha)$ can be expressed in algebraic form (e.g., the constraint function $f(\cdot,\cdot)$ is linear in both arguments and the random data vector is Gaussian \cite{bienstock}) or cases in which the cumulative distribution $F_{f(x,\Xi)}(\cdot)$ and its derivatives can be computed explicitly \cite{henrion}. Exact reformulations to its sample average approximation (SAA) with integer variables, originally proposed in \cite{luedtkeahmed}, use an indicator function representation of the CC. Unfortunately, in the context of CC-P, the integer reformulation would lead to large-scale and nonconvex mixed-integer nonlinear programs (MINLPs). Conservative and computationally more tractable approximations of CC-P can be used to avoid the need for solving MINLPs.  A conservative approximation can be obtained by using the so-called scenario-based approach \cite{campi,calafiore2006scenario,nemirovski2006scenario}. In this approach, we solve a stochastic NLP that enforces $f(x,\Xi)\leq 0$ with probability one (almost surely). Such an approach leads to structured NLPs, which can in turn be solved using parallel interior-point solvers \cite{zavalalaird}.  A drawback of the scenario approach is that it can be overly conservative and does not offer direct control on the probability level of CC.  Alternative conservative approximations include the conditional value-at-risk (CVaR) approximation and the Bernstein approximation, which use convex approximations of the indicator function \cite{nemirovski2006convex}. The level of conservatism of both CVaR approximation and Bernstein approximation might be reduced, but not eliminated using the ``tuning methods" \cite{nemirovski2006convex}. The authors in \cite{hong2011sequential,shan2014smoothing} propose a difference of convex functions (DC) approximation for the indicator function and they show that the approximation can be made equivalent to CC-P for limiting parameters. This approach involves a difference of non-smooth max functions that cannot be handled with standard NLP modeling and solution tools. Instead, this approach requires specialized solution algorithms that are not guaranteed to work in a general nonconvex NLP setting. 

The authors in \cite{geletu2017inner,geletu2015tractable} propose a smooth sigmoidal (SS) approximation for the indicator function. The solutions of this approximation are shown to be conservative and converge to the solutions of CC-P for limiting parameter values. This approach has the practical advantage that its sample average approximation can be handled using standard NLP tools. Unfortunately, no guidelines have been provided to select suitable approximation parameters. This is important, because, when the parameter values are far away from their limiting values, the approximation can be very conservative and lead to infeasible problems. On the other hand, when the parameter values are too close to their limiting values, the sigmoidal approximation becomes difficult to handle numerically. 

In this work, we propose a tailored sigmoidal approximation to outer-approximate the indicator function.  We use this sigmoidal function to construct a risk measure, that we call SigVaR, and show that this is a conservative approximation of the value at risk (VaR). We prove that the SigVaR approximation is always conservative and that it converges to CC-P for limiting parameter values.  As with SS, a benefit of the SigVaR approximation is that it can be handled by using standard NLP solvers, thus offering parallel solution capabilities. We establish explicit relationships between the parameters of SigVaR with those of CVaR. This allows us to establish parameter values that guarantee that the approximation is as conservative as CVaR.  Specifically, we show that we can directly relate the parameters of the SigVaR approximation to the value-at-risk (VaR) identified with CVaR. This connection provides a mechanism to obtain an initial feasible solution and an initial guess for the parameter values.  We also establish explicit connections between the parameters of SigVaR and those of DC and SS approximations. As with SS, a drawback of SigVaR is that numerical stability is encountered as the approximation approaches the indicator function. To improve this issue, we propose a scheme that solves a sequence of conservative approximations of increasing quality. Another drawback of SigVaR is that solving SigVaR to global optimality is computationally intractable if the dimension of $x$ is large. Scenario-based approach, CvaR approximation, DC approximation, and and SS approximation all have the same drawback if nonconvex functions are involved. Actually Even solving the original large-scale NLP problem without chance constraints to global optimality is computationally intractable. However, solving SigVaR to local optimality also provides promising performance for many problems as shown in Section \ref{sec:results}. Small and large case studies are used to illustrate the concepts and demonstrate performance. 

The paper is organized as follows. Section \ref{sec:othermethod}  introduces basic nomenclature and reviews CVaR, Bernstein, and DC approximations. Section \ref{sec:sigmoid} introduces the SigVaR approximation and establishes properties. Section \ref{sec:comp} outlines a numerical scheme to solve a sequence of SIgVaR approximations. Section \ref{sec:results} provides numerical studies. Final remarks are provided in Section \ref{sec:final}.

\section{Review on CC Approximations}\label{sec:othermethod}

We review approaches to deal with CC-P in order to introduce some necessary concepts and notation. Derivations follow the work of \cite{nemirovski2006convex,pinter1989deterministic}.  We make the following blanket assumptions throughout the paper. 
\begin{asm}
\label{asm:cont}
There exists $\kappa > 0$ such that $F_{Z(x)}(t)$ is Lipschitz continuous in $t \in [-\kappa, \kappa]$ for every $x \in  \mathcal{X}$.
\end{asm}
Assumption \ref{asm:cont} is slightly weaker than the Assumption 4 in
 \cite{hong2011sequential}. We will discuss special cases violating Assumption \ref{asm:cont}, that is When $Z(x)$ is a discrete random variable, with $\mathbb{P}(Z(x) = 0)  = 0$ for every $x \in  \mathcal{X}$, in Section \ref{subsec:cc-p} following Theorem \ref{th:converge}.

\begin{asm}
\label{asm:reg}
$\mathcal{X}(\alpha) =  cl \mathcal{X}^I(\alpha)$ with $\mathcal{X}^I(\alpha):=\{\,x \in \mathcal{X}\,|  \mathbb{P}\left(f(x,\Xi)> 0\right) < \alpha\}$.
\end{asm}
This regularity assumption is Assumption 5 in \cite{hong2011sequential}.

\subsection{CVaR Approximation}
Because $\mathbb{P}(Z> 0)  = \mathbb{E}[1_{(0, \infty)}(Z)]$, the CC can be expressed as $\mathbb{P}(f(x,\Xi)> 0) \leq \alpha$, and we can use the equivalent formulation:
\begin{align}
\mathbb{E}[1_{(0, \infty)}(f(x,\Xi))] \leq \alpha. 
\end{align}
A computationally practical approach to approximate the CC is to find a {\em conservative approximation}. This is done by finding an approximating function $\psi:\mathbb{R}\to \mathbb{R}$ satisfying  $\psi(z) \geq 1_{[0, \infty)}(z) \geq 1_{(0, \infty)}(z)$ for any $z\in\mathbb{R}$. For such a function we have that $\psi(t^{-1}z) \geq 1_{[0, \infty)}(t^{-1}z) \geq 1_{(0, \infty)}(z)$ for any parameter $t>0$.  Consequently,
\begin{align}
\mathbb{E} [\psi(t^{-1}Z)] \geq \mathbb{P}(Z > 0).
\end{align}
We can thus conclude that the satisfaction of the constraint:
\begin{align}
\mathbb{E} [\psi(t^{-1}Z)] \leq \alpha,\label{essential_eqns}
 \end{align}
implies that $\mathbb{P}(Z> 0) \leq \alpha$ is satisfied (and so does $\mathbb{P}\left(f(x,\Xi)\leq 0\right) \geq 1- \alpha$). Because \eqref{essential_eqns} is valid for all $t>0$ we also have, if $\psi(\cdot)$ is convex, that:
\begin{align} 
\inf_{t>0} \{ t\,\mathbb{E} [\psi(t^{-1}Z)] - t\,\alpha\} \leq 0  \label{essential_inf_eqns}
\end{align}
implies $\mathbb{P}(Z> 0) \leq \alpha$. The quality of the conservative approximation depends on the choice of the approximating function $\psi(\cdot)$.  The choice  $\psi(z) := [1+z]_{+}$ with $[z]_{+}:=\text{max}\{z,0\}$ leads to the approximation:
\begin{align} \label{eq:cvar1}
\inf_{t>0} \{ \mathbb{E} \left[[Z+t]_{+}\right] - t\alpha\} \leq 0. 
\end{align}
$\inf_{t>0}$ can be replaced with $\inf_{t}$ to obtain:
\begin{align} 
\inf_{t\in \mathbb{R}} \left\{ \alpha^{-1} \mathbb{E} \left[[Z+t]_{+}\right] -t\right\} \leq 0. \label{quasicvar_eqns}
\end{align}
By redefining $t\leftarrow -t$ and recalling that $\text{CVaR}_{1-\alpha}(Z) {:=}  \inf_{t\in \mathbb{R}} \left\{ t + \alpha^{-1} \mathbb{E} \left[[Z-t]_{+}\right] \right\}$, we see that \eqref{quasicvar_eqns} can be used to derive a conservative approximation of CC-P of the form:
\begin{subequations}
\label{eq:cvarprob}
\begin{align}
\min \limits_{x \in  \mathcal{X} } \;\;  &\varphi(x)  \label{cvar1} \\
{\rm s.t.}  \;\; & \text{CVaR}_{1-\alpha}(f(x,\Xi)) \leq 0. \label{cvar_formulation}
\end{align}
\end{subequations}
We denote an optimal objective value and solution of this problem (which we call CVaR-P) as $\varphi_c(\alpha)$ and $x_c (\alpha)$, respectively. We  define the feasible set of CVaR-P as $\mathcal{X}_c(\alpha)$ and notice, because CVaR provides a conservative approximation,  that $\mathcal{X}_{c}(\alpha)\subseteq \mathcal{X}(\alpha)$. This also implies that $\varphi_c(\alpha)\geq \varphi(\alpha)$ for all $\alpha\in (0,1]$.


A key advantage of the CVaR approximation is that its sample average approximation (SAA) can be cast as a standard NLP \cite{cao2017scalable, rockafellar2000optimization}. Moreover, if $f(x, \xi)$ is convex in $x$ for given $\xi$, CVaR is also convex in $x$. One can also prove that the function $\psi(z) = [1+z]_{+}$ is the {\em tightest convex approximation} of $1_{[0, \infty)}(z)$.  Despite these benefits, the CVaR approximation can be quite conservative. Moreover, the CVaR approximation does not offer a mechanism to enforce convergence to a solution of CC-P. 

\subsection{Bernstein Approximation}
If we use  the function $\psi(z) = e^{z}$, \eqref{essential_eqns} takes the form $\mathbb{E} [e^{t^{-1}Z}]  \leq \alpha$. For $t >  0$ this is equivalent to,
\begin{align}
t\, \text{log}(\mathbb{E} [e^{t^{-1}Z}]) \leq t \,\text{log} (\alpha ).
 \end{align}
Because this relationship is valid for all $t>0$, we can also conclude that:
\begin{align} 
\inf_{t>0} \{ t\, \text{log}(\mathbb{E} [e^{t^{-1}Z}]) - t\, \text{log} (\alpha )  \} \leq 0 \label{Bernstein_eqns}
 \end{align}
which is called Bernstein approximation. From the definition of entropic value-at-risk (EVaR) \cite{ahmadi2012entropic}:
\begin{align} 
 \text{EVaR}_{1-\alpha}(Z):= \inf_{ t>0} \left\{ t^{-1} \text{log}\left(\alpha^{-1}\mathbb{E} [e^{t Z}]\right)\right\},  \label{evar_definition}
\end{align}
it is thus easy to see that \eqref{Bernstein_eqns} is equivalent to:
\begin{align} 
 \text{EVaR}_{1-\alpha}(Z) \leq 0 \label{evar_formulation}.
\end{align}
This conservative approximation can be handled using standard NLP techniques. Moreover, if $f(x, \Xi)$ is convex in $x$ for given $\xi$, EVaR is also convex in $x$. Unfortunately, one can prove that  EVaR is even more conservative than CVaR. This follows from $\text{VaR}_{1-\alpha}(Z)\leq \text{CVaR}_{1-\alpha}(Z)\leq\text{EVaR}_{1-\alpha}(Z)$ \cite{nemirovski2006convex}.

\subsection{DC Approximation}

In \cite{hong2011sequential} it is shown that the indicator function can be approximated by using a difference of convex functions (the authors in  \cite{hong2011sequential} assume assume that $f$ is convex). The DC approximation of $\mathbb{P}(f(x,\Xi)>0)\leq \alpha$ has the form:
\begin{align}
\epsilon^{-1}\mathbb{E}\left[\psi(f(x,\Xi),\epsilon)-\psi(f(x,\Xi),0)\right]\leq \alpha.\label{eq:eps}
\end{align} 
where $\psi(z,t):=[z+t]_+$, $\epsilon\in\mathbb{R}_+$ is an approximation parameter. By using approximation \eqref{eq:eps} instead of \eqref{ccprob2}, we obtain problem DC-P.  In \cite{hong2011sequential} it is shown that DC-P is equivalent to CC-P for $\epsilon \to 0$. A practical limitation of DC-P is that its SAA cannot be cast as a standard NLP, due to the difference of max functions. Consequently, tailored algorithms are needed \cite{hong2011sequential}.


\subsection{Smooth Sigmoidal Approximation}
The authors in \cite{geletu2015tractable} approximate the indicator function by using a smooth sigmoidal function. The approximation has the form: 
\begin{align}
\mathbb{E}\left[\psi_{sm}^{\rho,m_1,m_2}(f(x,\Xi))\right]\leq \alpha \label{smoothprob}
\end{align}
where $\psi_{sm}^{\rho,m_1,m_2}(z):=\frac{1+\rho m_1}{1+\rho m_2e^{-z/\rho }}$, $m_1,m_2,\rho\in\mathbb{R}_+$ are approximation parameters satisfying $0 <m_2 \leq m_1$ and $\rho>0$. The framework proposed in \cite{geletu2015tractable} also consider the possibility of using functions $m_1(x)$ and $m_2(x)$. The SS approximation is exact in the limit $\rho \to 0$. We denote an optimal objective value and the feasible set of the  approximation with \eqref{smoothprob} (which we call SS-P) as $\varphi^{\rho,m_1, m_2}_{sm}(\alpha)$ and $\mathcal{X}^{\rho,m_1, m_2}_{sm}(\alpha)$. An important practical limitation of this approximation is that no guidelines exist to choose $\rho, m1, m2$.

\section{SigVaR Approximation} \label{sec:sigmoid}

As noticed in \cite{geletu2015tractable}, our work is motivated by the observation that the indicator function can be outer-approximated by using a {\em standard sigmoid function} of the form:
\begin{equation}\label{eq:sig}
\psi_s^{\mu,\tau}(z) := \frac{1+\mu}{ \mu + e^{-\tau z}},
\end{equation}
where $\mu,\tau\in\mathbb{R}_+$ are the approximation parameters. The associated CC approximation takes the form:
\begin{align} 
 \mathbb{E} \left[\psi_s^{\mu,\tau}(f(x,\Xi))\right]  \leq \alpha. \label{sigmoid_eqns}
\end{align}
The sigmoid function \eqref{eq:sig} is equivalent to $\psi_{sm}^{\rho,m_1,m_2}(z)$ when $\tau = \frac{1}{\rho}$, $\mu =   \frac{1}{\rho m_2}$, and $m_1=m_2$. The sigmoid function is also a special case of the generalized logistic function, which is a standard approximation function for the indicator function \cite{chen1995smoothing}. 

In this work, we consider a {\em variant of the sigmoid function} \eqref{eq:sig} of the form:
\begin{equation}
\psi_{ss}^{\mu,\tau}(z) := \left[2\frac{1+\mu}{ \mu + e^{-\tau z}}-1\right]_{+}. \label{eq:sig_va}
\end{equation}
This gives the  CC approximation:
\begin{align} 
 \mathbb{E}\left[ \psi_{ss}^{\mu,\tau}(f(x,\Xi))\right] \leq \alpha. \label{sigmoid_variant}
\end{align}
Although $\psi_{ss}^{\mu,\tau}$ is non-smooth, in Section \ref{sec:comp} we show that the sample average approximation with constraint \eqref{sigmoid_variant} can be cast as a standard NLP.

The motivation behind the tailored variant is illustrated in Figure \ref{fig.svar}, where we can see that the variant is more accurate than the standard counterpart because the max function sets $\psi_{ss}^{\mu,\tau}(z)=0$ for all $z \leq - \delta$ {where $\delta := \frac{1}{\tau} \log(2+\mu)$}. Although \eqref{eq:sig_va} is not smooth, we show in Section \ref{sec:comp} that it can still be cast as a standard NLP. In the following sections we prove that the structure of the proposed variant allows us to establish connections with CVaR and DC approximations. 

\begin{figure}[tbhp]
\begin{center}
\includegraphics[width=0.6\textwidth]{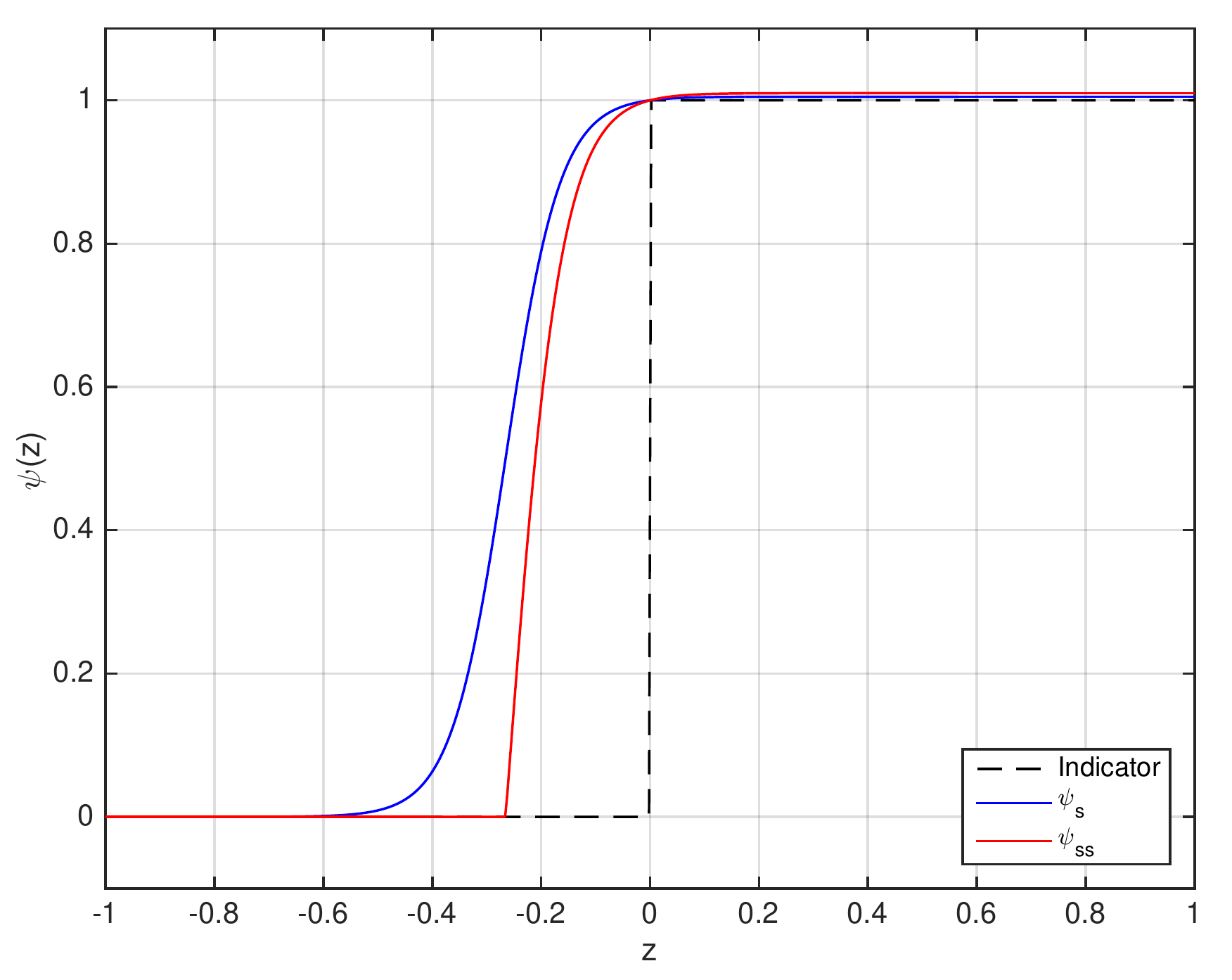}\vspace{-0.1in}
\caption{Comparison of standard and tailored sigmoid functions.}
\label{fig.svar}
\end{center}
\end{figure}

We being by showing that sigmoid functions provide natural conservative approximations for CCs. 
\begin{lemma}
\label{thm:conservative}
The constraints \eqref{sigmoid_eqns} and \eqref{sigmoid_variant} are conservative approximations of the CC \eqref{ccprob2} for any $\mu,\tau\in \mathbb{R}_+$, $\alpha \in (0,1]$, and $x\in \mathcal{X}$.
\end{lemma}
\begin{proof}
Consider the random variable $Z=f(x,\Xi)$ with realizations $z\in \mathbb{R}$. Since $e^{-\tau z} \geq 0$ holds for $z\in\mathbb{R}$ and $e^{-\tau z} \leq 1$ holds for $z\in\mathbb{R}_+$ we have that $ \frac{1+\mu}{ \mu + e^{-\tau z}}\geq0$ for any $z\in\mathbb{R}$ and $ \frac{1+\mu}{ \mu + e^{-\tau z}}\geq1$ holds for $z\in\mathbb{R}_+$. We thus have that $ 2\frac{1+\mu}{ \mu + e^{-\tau z}} - 1 \geq1$ holds for $z\in\mathbb{R}_+$. Therefore, $ \frac{1+\mu}{ \mu + e^{-\tau z}} \geq 1_{[0, \infty)}(z)$ and $\psi_{ss}^{\mu,\tau}(z)  \geq 1_{[0, \infty)}(z)$  for any $z\in\mathbb{R}_+$. Consequently,  $\mathbb{E} [\frac{1+\mu}{ \mu + e^{-\tau Z}}]\geq \mathbb{E} [1_{[0, \infty)}(Z)] \geq \mathbb{P}(Z> 0)$ and  $\mathbb{E}\left[ \psi_{ss}^{\mu,\tau}(Z) \right]\geq \mathbb{E} [1_{[0, \infty)}(Z)] \geq \mathbb{P}(Z> 0)$. The result follows.  $\Box$
\end{proof}


We use the proposed function to define the Sigmoidal Value-at-Risk (SigVaR): 
\begin{align}
\text{SigVaR}_{1-\alpha}^{\mu,\tau}(Z)&:=\inf \left\{ t\in\mathbb{R}: \mathbb{E} \left[ \psi_{ss}^{\mu,\tau}(Z-t)\right]  \leq \alpha \right\}\nonumber\\
&\,=\inf \left\{  t\in\mathbb{R} : \mathbb{E} \left[ \left[2\frac{1+\mu}{ \mu + e^{-\tau(Z-t)}}-1\right]_{+}\right]  \leq \alpha \right\}.
\end{align}
and we use this to formulate the problem:
\begin{subequations}
\label{eq:svarprob}
\begin{align}
\min \limits_{x \in  \mathcal{X} } \;\;  &\varphi(x)  \\
{\rm s.t.}  \;\; & \text{SigVaR}_{1-\alpha}^{\mu,\tau}(f(x,\Xi)) \leq 0. \label{svar_formulation}
\end{align}
\end{subequations}
We define an optimal objective and solution of \eqref{eq:svarprob} as $\varphi_{ss}^{\mu,\tau}(\alpha)$ and $x_{ss}^{\mu,\tau}(\alpha)$. We also denote the set of optimal solutions as $S_{ss}^{\mu,\tau}(\alpha)$ and define the feasible set of \eqref{eq:svarprob} as $\mathcal{X}_{ss}^{\mu,\tau}(\alpha)$. From Lemma \ref{thm:conservative}, it is clear that $\mathcal{X}_{ss}^{\mu,\tau}(\alpha)\subseteq\mathcal{X}(\alpha)$ for all $\mu,\tau\in \mathbb{R}_+$.  This implies that $\varphi_{ss}^{\mu,\tau}(\alpha)\geq \varphi(\alpha)$ for all $\alpha\in (0,1]$ and $\mu,\tau\in \mathbb{R}_+$. 

The definition of SigVaR is motivated by the observation that $\text{VaR}_{1-\alpha}(Z)=\inf\{t: \mathbb{P}(Z-t > 0)\leq \alpha\}$ can also be expressed in terms of the indicator function:
\begin{align}
\text{VaR}_{1-\alpha}(Z)=\inf\{t\in \mathbb{R} : \mathbb{E}\left[1_{(0, \infty)}(Z-t)\right] \leq \alpha\}.
\end{align}
Because we have established that the sigmoid function $\psi_{ss}^{\mu,\tau}(\cdot)$ is a conservative approximation of $1_{(0, \infty)}(\cdot)$, we have that $\text{SigVaR}_{1-\alpha}^{\mu,\tau}(Z)\geq \text{VaR}_{1-\alpha}(Z)$. Consequently, SigVaR can be interpreted as an approximate quantile and \eqref{eq:svarprob} is a conservative representation of CC-P.  As in the case of the VaR representation of CC-P, problem \eqref{eq:svarprob} is not particularly attractive for computation. However, this problem also has the following equivalent representation (that we call SigVaR-P):
\begin{subequations}
\label{eq:svarprob2}
\begin{align}
\min \limits_{x \in  \mathcal{X} } \;\;  &\varphi(x)  \\
{\rm s.t.}  \;\; &  \mathbb{E}\left[ \psi_{ss}^{\mu,\tau}(f(x,\Xi))\right] \leq \alpha. \label{eq:approxcc}
\end{align}
\end{subequations} 
In Section \ref{sec:comp} we will show that the SAA approximation of SigVaR-P can be cast as a standard NLP. 

To show that \eqref{eq:svarprob2} and \eqref{eq:svarprob} are equivalent, we make the following observations. If $\mathbb{E} \left[ \psi_{ss}^{\mu,\tau}(Z)\right] \leq \alpha$ is satisfied then it implies that $t=0$ satisfies $\mathbb{E} \left[ \psi_{ss}^{\mu,\tau}(Z-t)\right] \leq \alpha$,  and since $\text{SigVaR}_{1-\alpha}^{\mu,\tau}(Z)$ is the smallest $t$ satisfying $\mathbb{E} \left[ \psi_{ss}^{\mu,\tau}(Z-t)\right] \leq \alpha$, then $\text{SigVaR}_{1-\alpha}^{\mu,\tau}(Z) \leq 0$. On the other hand, if $\text{SigVaR}_{1-\alpha}^{\mu,\tau}(Z) \leq 0$ is satisfied, according to the definition, $t=\text{SigVaR}_{1-\alpha}^{\mu,\tau}(Z)$ satisfies $\mathbb{E} \left[ \psi_{ss}^{\mu,\tau}(Z-t)\right]  \leq \alpha$. Since  $\mathbb{E} \left[ \psi_{ss}^{\mu,\tau}(Z-t)\right]$ is a decreasing function of $t$, then $t=0$ also satisfies $\mathbb{E} \left[ \psi_{ss}^{\mu,\tau}(Z-t)\right]  \leq \alpha$ and thus $\mathbb{E} \left[ \psi_{ss}^{\mu,\tau}(Z)\right] \leq \alpha$.

\subsection{Relationship with CC-P}\label{subsec:cc-p}

We now show that SigVaR-P becomes an exact approximation of CC-P in the limit of its parameter values. For the random variable $Z(x)=f(x,\Xi)$ with $x\in \mathcal{X}$, we define the SigVaR-CC approximation error:
\begin{align} 
\epsilon_{\mu,\tau}(x) &:=  \mathbb{E}\left[ \psi_{ss}^{\mu,\tau}(Z(x))\right] - \mathbb{P}(Z(x)>0). 
\end{align}
From Lemma \ref{thm:conservative} we have that $\epsilon_{\mu,\tau}(x)\geq 0$ for all $\mu,\tau\in \mathbb{R}_+$.

We proceed to establish a bound for the SigVaR-CC approximation error. Under Assumption \ref{asm:cont} we can establish that there exists a positive constant $L(x)$ satisfying $\mathbb{P}(-\nu\leq Z(x) \leq 0)  \leq L(x)\nu$ for all $x\in\mathcal{X}$, and any $\nu \in \mathbb{R}_+$. The reasoning is the following: when $\nu \leq \kappa$, from Lipschitz continuity of $F_{Z(x)}(t)$, we get $\mathbb{P}(-\nu\leq Z(x) \leq 0) = F_{Z(x)}(0) - F_{Z(x)}(-\nu) \leq L(x)\nu$, where $L(x)$ is set to be the Lipschitz constant.  When $\nu > \kappa$, we have $\mathbb{P}(-\nu\leq Z(x) \leq 0) \leq 1 \leq \frac{\nu}{\kappa}$ and we have $L(x)= \frac{1}{\kappa}$. A special case satisfying Assumption \ref{asm:cont} is when $Z(x)$ is a continuous random variable with bounded probability density $p_{Z(x)} < \infty$. In this case, we have that constant $L(x)=  \sup_{z\in \mathbb{R}}\{p_{Z(x)}(z)\}\in (0,\infty)$ exists and satisfies $\mathbb{P}(-\nu\leq Z(x) \leq 0) =  \int_{-\nu}^{0} p_{Z(x)}(z) dz \leq \int_{-\nu}^{0} L(x) dz=L(x)\nu$ for all $x\in\mathcal{X}$.

\begin{lemma}
\label{bounderror}
The SigVaR-CC error is bounded as $\epsilon_{\mu,\tau}(x) \leq\frac{\log(2+\mu)L(x)}{\tau} +  \frac{2}{\mu} $ for all $x\in \mathcal{X}$. 
\end{lemma}

\begin{proof}  
We establish the result by following sequence of implications:
\begin{align*} 
\epsilon_{\mu,\tau}(x) &=   \mathbb{E}\left[ \psi_{ss}^{\mu,\tau}(Z(x))\right]  -  \mathbb{E}\left[1_{(0, \infty)}(Z(x))\right]  \\
&=   \mathbb{E}\left[ \psi_{ss}^{\mu,\tau}(Z(x)) - 1_{(0, \infty)}(Z(x))\right]  \\
&=   \mathbb{E}\left[ \psi_{ss}^{\mu,\tau}(Z(x)) | Z(x) < - \frac{1}{\tau} \log(2+\mu)\right]  \mathbb{P}\left(Z(x) < - \frac{1}{\tau} \log(2+\mu)\right)  \\
&+ \mathbb{E}\left[ \psi_{ss}^{\mu,\tau}(Z(x)) | - \frac{1}{\tau} \log(2+\mu) \leq Z(x) \leq 0\right]  \mathbb{P}\left(- \frac{1}{\tau} \log(2+\mu)  \leq Z(x) \leq 0\right)  \\
& +   \mathbb{E}\left[\psi_{ss}^{\mu,\tau}(Z(x)) - 1 | Z(x) > 0 \right]  \mathbb{P}(Z(x) > 0)   \\
& \leq \mathbb{E}\left[ 1 | - \frac{1}{\tau} \log(2+\mu) \leq Z(x) \leq 0\right]  \mathbb{P}\left(- \frac{1}{\tau} \log(2+\mu)  \leq Z(x) \leq 0\right)  \\
& +   \mathbb{E}\left[ \frac{2}{\mu} | Z(x) > 0 \right]  \mathbb{P}(Z(x) > 0)   \\
& =   \mathbb{P}\left(- \frac{1}{\tau} \log(2+\mu)  \leq Z(x) \leq 0\right)  +   \frac{2}{\mu}   \mathbb{P}(Z(x) > 0)   \\
&\leq \frac{1}{\tau}{\log(2+\mu)L(x)}+  \frac{2}{\mu}. 
\end{align*}
Here, the first inequality follows since $\psi_{ss}^{\mu,\tau}(Z(x))=0$ for $Z(x) < - \frac{1}{\tau} \log(2+\mu)$, $\psi_{ss}^{\mu,\tau}(Z(x)) \leq 1$ for $ - \frac{1}{\tau} \log(2+\mu)  \leq Z(x) \leq 0$ and $\psi_{ss}^{\mu,\tau}(Z(x)) \leq  \frac{2}{\mu} + 1$ for $  Z(x) >  0$. The last inequality follows from $ \mathbb{P}( Z > 0)\leq 1$. $\Box$
\end{proof}


\begin{thm}
\label{th:converge}
Let $\tau(\mu):= (1+\mu) \theta$  with $\theta >0$. Then $ \lim\limits_{\mu\to \infty} \mathbb{E}\left[ \psi_{ss}^{\mu,\tau}(Z(x))\right] = \mathbb{P}(Z(x)>0)$.
\end{thm}
\begin{proof}
From Lemma \ref{bounderror} we can establish the bound {$\epsilon_{\mu,\tau} \leq \tau(\mu)^{-1}{\log(2+\mu)}L +  2\mu^{-1}$} with $L:=\sup_{x\in\mathcal{X}}L(x)$. The result follows. $\Box$
\end{proof}

{\bf Remark:} When $Z(x)$ is a discrete random variable, we can establish the error bound of Lemma \ref{bounderror} if $Z(x)$ has finite outcomes and we have that $\mathbb{P}(Z(x) = 0)  = 0$. Here, we assume that $Z(x)$ has finite $m$ possible outcomes $z_1(x)< z_2(x) <\dots < z_{m'}(x)<0<z_{m'+1}(x)<\dots<z_m(x)$ with corresponding probabilities as $p_{i}(x),\,i=1,...,m$. A bounding constant $L(x)$ can be found in this case by noticing that $\mathbb{P}(-\nu\leq Z(x) \leq 0) = \sum_{i=1}^{m'} p_i(x)$ if $-\nu \leq z_1(x)$, $\mathbb{P}(-\nu\leq Z(x) \leq 0) = \sum_{i=k}^{m'} p_i(x)$, if  $z_{k-1}(x)< -\nu \leq z_k(x)$,  and $\mathbb{P}(-\nu\leq Z(x) \leq 0) = 0$ if  $z_{m'}(x)< -\nu$. We thus have that $L(x):= \max_{k\in \{ 1, \dots, m' \}} \left\{ {\sum_{i=k}^{m' } p_i(x)}/{z_k(x)}   \right\}$ satisfies $\mathbb{P}(-\nu\leq Z(x) \leq 0) \leq L(x)\nu $. Consequently, the results of Theorem \ref{th:converge} hold. This is relevant because we are often interested in solving discrete approximations of SigVar-P (e.g., by using SAA). 


The following result shows that we can construct a sequence of SigVaR approximations of increasing quality by progressively increasing $\mu$. 
{
\begin{lemma}
Let  $\tau(\mu):= (1+\mu) \theta$ with $\theta>0$. We have that $\mathcal{X}_{ss}^{\mu^+,\tau(\mu^+)}(\alpha)  \supseteq \mathcal{X}_{ss}^{\mu,\tau(\mu)}(\alpha)$ and $ \varphi_{ss}^{\mu^+,\tau(\mu^+)}(\alpha)\leq \varphi_{ss}^{\mu,\tau(\mu)}(\alpha)$  for $\mu^+>\mu>0$ and for all $\alpha \in (0,1]$.  
\label{thm:monotonical}
\end{lemma}
\begin{proof}
We show that $\psi_{ss}^{\mu^+,\tau}(z) < \psi_{ss}^{\mu,\tau}(z)$ for any $z\in \mathbb{R} \setminus \{0\}$  (for $z=0$, we have $\psi_{ss}^{\mu,\tau}(z)=1$ for any $\mu$). To proceed, it suffices to show that the kernel function $\frac{1+\mu}{ \mu + e^{-\tau(\mu) z}}$ is a strictly decreasing function of $\mu$ for all $z\in \mathbb{R} \setminus \{0\}$. We establish this by showing that the derivative of of the kernel function is negative:
\begin{align*} 
\frac{d}{d\mu}\left({\frac{1+\mu}{ \mu + e^{-\tau(\mu) z}} }\right) &=  \frac{\mu + e^{-\tau(\mu) z}  - (1+\mu) (1  -  \theta z  e^{-\tau(\mu) z} )}{ (\mu + e^{-\tau(\mu) z})^2  } \\
		&= \frac{-1  +   (1+   (1+\mu)\theta z    )  e^{-\tau(\mu) z}         }{ (\mu + e^{-\tau(\mu) z})^2  } \\ 
		&= \frac{-1  +   (1+   \tau(\mu) z    )  e^{-\tau(\mu) z}         }{ (\mu + e^{-\tau(\mu) z})^2  } \\ 
       & < 0 .      
\end{align*}
The last step follows from {$1+\tau(\mu) z  < e^{\tau(\mu) z}$}, for any $z\in\mathbb{R} \setminus \{0\}$ (from Taylor's theorem and from the convexity of the exponential function).$\Box$
\end{proof}
}
The following result establishes convergence of the feasible set of SigVaR-P to that of CC-P. 

\begin{thm}
Let  $\tau(\mu) := (1+\mu) \theta$ with $\theta>0$. We have $\lim\limits_{\mu\to \infty}\mathcal{X}_{ss}^{\mu,\tau(\mu)}(\alpha){=}\mathcal{X}(\alpha)$.  
\label{thm:feasibleregionconverge}
\end{thm}
\begin{proof}
Take an arbitrary increasing sequence $\{\mu_k\}_{k \in \mathbb{N}}$ with $\mu_k \to \infty$. From Lemma \ref{thm:monotonical} and Exercise 4.3 (a) of \cite{rockafellar2009variational}, $\lim\limits_{k \to \infty}\mathcal{X}_{ss}^{\mu_k,\tau(\mu_k)}(\alpha)$ exist and $\lim\limits_{k \to \infty}\mathcal{X}_{ss}^{\mu_k,\tau(\mu_k)}(\alpha)$ $= cl \cup_{k \in \mathbb{N}}  \mathcal{X}_{ss}^{\mu_k,\tau(\mu_k)}$. Since $| \psi_{ss}^{\mu_k,\tau_k}(Z(x))| < 1+\frac{2}{\mu_k}$, from Theorem 7.43 of \cite{shapiro2009lectures}, $\mathbb{E}\left[ \psi_{ss}^{\mu_k,\tau(\mu_k)}(Z(x))\right]$ is a continuous function of $x$ and thus $\mathcal{X}_{ss}^{\mu_k,\tau(\mu_k)}$ is a closed set. Lemma \ref{thm:conservative} implies $\lim\limits_{k \to \infty}\mathcal{X}_{ss}^{\mu_k,\tau(\mu_k)}(\alpha)$ $\subset \mathcal{X}(\alpha)$. 

We then prove that $\lim\limits_{k \to \infty}\mathcal{X}_{ss}^{\mu_k,\tau(\mu_k)}(\alpha) \supset \mathcal{X}(\alpha)$. For any $x \in \mathcal{X}^I(\alpha)$, because of Theorem \ref{th:converge} and $\mathbb{P}\left(f(x,\Xi)> 0\right) < \alpha$, then there exist $\mu_0$ so that all $\mu > \mu_0$, $\mathbb{E}\left[ \psi_{ss}^{\mu,\tau(\mu)}(Z(x))\right] < \alpha$ and thus $x \in \mathcal{X}_{ss}^{\mu,\tau(\mu)}(\alpha)$. Combined with $\lim\limits_{k \to \infty}\mathcal{X}_{ss}^{\mu_k,\tau(\mu_k)}(\alpha)$ $= cl \cup_{k \in \mathbb{N}}  \mathcal{X}_{ss}^{\mu_k,\tau(\mu_k)}$, we have $\lim\limits_{k \to \infty}\mathcal{X}_{ss}^{\mu_k,\tau(\mu_k)}(\alpha) \supset \mathcal{X}^I(\alpha)$ . Because $\lim\limits_{k \to \infty}\mathcal{X}_{ss}^{\mu_k,\tau(\mu_k)}(\alpha)$ is a closed set then, by Assumption \ref{asm:reg}, we have $\lim\limits_{k \to \infty}\mathcal{X}_{ss}^{\mu_k,\tau(\mu_k)}(\alpha)$ $\supset \mathcal{X}(\alpha)$. $\Box$
\end{proof}

The following is our main result, which establishes convergence of the solution set and optimal objective value.

\begin{thm}
Let  $\tau(\mu) := (1+\mu) \theta$ with $\theta>0$. We have $\lim\limits_{\mu\to \infty} \varphi_{ss}^{\mu,\tau(\mu)}(\alpha) = \varphi^*(\alpha)$ and $\limsup\limits_{\mu\to \infty} S_{ss}^{\mu,\tau(\mu)}(\alpha) \subset S^*(\alpha)$. 
\label{thm:objconverge}
\end{thm}
\begin{proof}
Let $\bar{\varphi}(x) = \varphi(x) + I_{\mathcal{X}(\alpha)}(x)$ and $\bar{\varphi}_{\mu}(x) = \varphi(x) + I_{\mathcal{X}_{ss}^{\mu,\tau(\mu)}}(x)$, where $I_A(x) = 0$ if $x \in A$ and $I_A(x) = + \infty$ if $x \notin A$. By Proposition 7.4(f) of \cite{rockafellar2009variational}, we have $I_{\mathcal{X}_{ss}^{\mu,\tau(\mu)}}(.)$ epi-converges to $I_{\mathcal{X}(\alpha)}(.)$ as  $\mu\to \infty$. Since $\varphi(x)$ is continuous, by Exercise 7.8(a) in \cite{rockafellar2009variational}, $\bar{\varphi}_{\mu}(.)$ epi-converges to $\bar{\varphi}(.)$ as  $\mu\to \infty$.  Because $\mathcal{X}(\alpha)$ is bounded, by Exercise 7.32 (a) of \cite{rockafellar2009variational}, $\bar{\varphi}_{\mu}(x)$ is eventually level bounded. Because $\bar{\varphi}_{\mu}(.)$ and $\bar{\varphi}(.)$ are lower semi-continuous and proper,  by Theorem 7.33 of \cite{rockafellar2009variational}, we have $\lim\limits_{\mu\to \infty} \varphi_{ss}^{\mu,\tau}(\alpha) = \varphi^*(\alpha)$ and $\limsup\limits_{\mu\to \infty} S_{ss}^{\mu,\tau}(\alpha) \subset S^*(\alpha)$.

$\Box$
\end{proof}

\subsection{Relationship with CVaR-P}

We define $Z_c(\alpha):=f(x_c(\alpha),\Xi)$ and recall that \cite{rockafellar2000optimization}:
\begin{align}
\text{VaR}_{1-\alpha}(Z_c(\alpha))= \argmin_{t} \left\{ t + \alpha^{-1} \mathbb{E} \left[[Z_c(t)-t]_{+}\right] \right\},
\end{align}
and thus $\text{VaR}_{1-\alpha}(Z_c(\alpha))\leq \text{CVaR}_{1-\alpha}(Z_c(\alpha))$.  This observation also highlights that CVaR provides a conservative approximation for the CC. 

Crucial to our results is the constant:
\begin{equation}\label{eq:afunc}
\gamma_\alpha:=-{t_c(\alpha)^{-1}}.
\end{equation}
with $t_c(\alpha)\in \argmin\limits_{t}\{ t + \alpha^{-1} \mathbb{E} [Z_c(t)-t]_{+}\}$.  

We now argue that we can always find a $t_c(\alpha)<0$ (equivalently $\gamma_\alpha>0$) at any $x_c(\alpha)$. 
Since \eqref{cvar_formulation} is satisfied at $x_c(\alpha)$ and $\text{VaR}_{1-\alpha}(Z_c(\alpha))\leq \text{CVaR}_{1-\alpha}(Z_c(\alpha))$, we have that either  $ \text{VaR}_{1-\alpha}(Z_c(\alpha)) <  0$ or $ \text{VaR}_{1-\alpha}(Z_c(\alpha)) =  \text{CVaR}_{1-\alpha}(Z_c(\alpha)) = 0$. In the first case, it follows that $\gamma_\alpha > 0$ with $t_c(\alpha)=\text{VaR}_{1-\alpha}(Z_c(\alpha))$. In the latter case, from $ \text{VaR}_{1-\alpha}(Z_c(\alpha)) =  \text{CVaR}_{1-\alpha}(Z_c(\alpha)) = 0$, we have $\mathbb{E} \left[[Z_c(\alpha)]_{+} \right]  = 0$, thus $\mathbb{P}(Z > 0) = 0$ and $F_{Z_c}(0) = 1$. We have both $F_{Z_c}(0) = 1$ and $F_{Z_c}(0) = 1-\alpha$ (from $ \text{VaR}_{1-\alpha}(Z_c(\alpha))=0$), which violates the Assumption \ref{asm:cont}. Thus the later case is not valid. Even if Assumption \ref{asm:cont} does not hold (e.g. $Z_c$ is a discrete random variable with finite outcomes, as long as $\mathbb{P}(Z = 0) = 0$, we cannot have that both $F_{Z_c}(0) = 1$ and $\text{VaR}_{1-\alpha}(Z_c(\alpha)) = 0$ hold.
 
We now show that the parameters of SigVaR-P can be selected in such a way that they provide an approximation of CC-P that is at least as good as that of CVaR-P. 

\begin{proposition}
Assume a fixed $\alpha\in (0,1]$ and that $\mu,\tau_{\alpha}\in\mathbb{R}_+$ satisfy $\mu\geq \bar{\mu}$ (where $\bar{\mu}\in\mathbb{R}_+$ is the positive solution of ${\bar{\mu}-\log(2+\bar{\mu})}=1$), $\tau_{\alpha}:=\frac{\mu+1}{2}\gamma_\alpha$, and $\gamma_\alpha$ defined in \eqref{eq:afunc}. We have that $\varphi_{ss}^{\mu,\tau_\alpha}(\alpha) \leq \varphi_c(\alpha)$. 
\label{thm:svarvscvar}
\end{proposition}

\begin{proof}
For simplicity, we omit dependency on $\alpha$ for $x_c(\alpha)$, $\gamma_\alpha$, and $\tau_{\alpha}$ (we simply write $x_c,\gamma,\tau$). We proceed by proving that any solution $x_c$ of CVaR-P is a feasible point for SigVaR-P provided that $\mu,\tau$ satisfy the conditions of the proposition. This would imply that we can always find $\mu,\tau$ such that $\varphi_{ss}^{\mu,\tau}(\alpha) \leq \varphi_c(\alpha)$. We define the random variable $Z_c=f(x_c,\Xi)$ with realizations $z_c\in \mathbb{R}$; the constraint \eqref{cvar_formulation} evaluated at $x_c,\gamma$ can be written as $ \mathbb{E} [[\gamma Z_c+1]_{+}]  \leq \alpha$. It suffices to show that $[\gamma z_c+1]_{+} \geq [2\frac{1+\mu}{ \mu + e^{-\tau z_c}} - 1]_{+}$ holds for any $z_c\in\mathbb{R}$. If $z_c < - \delta$, where $\delta := \frac{1}{\tau} \log(2+\mu)$, we have that $2\frac{1+\mu}{ \mu + e^{-\tau z_c}} - 1<0$  and, consequently,  $[\gamma z_c+1]_{+} \geq [2\frac{1+\mu}{ \mu + e^{-\tau z_c}} - 1]_{+}$.  For $z_c \geq - \delta$ we have that,
\begin{align} 
 \gamma z_c+1 \geq&  1 -  \frac{\gamma}{\tau} \log(2+\mu) \nonumber\\
\geq&  1 -  \frac{2\log(2+\mu)}{\mu+1}\nonumber\\ 
>&0. \label{eq::azpositive}
\end{align}
The last inequality follows because $\frac{2\log(2+\mu)}{\mu+1}$ is a monotonically decreasing function for $\mu\in\mathbb{R}_+$. We also observe that, for $2\frac{1+\mu}{ \mu + e^{-\tau z_c}} - 1\geq 0$,
\begin{align} 
 [\gamma z_c+1]_{+} - \left[2\frac{1+\mu}{ \mu + e^{-\tau z_c}} - 1\right]_{+} &=  [\gamma z_c+1] -  \left[2\frac{1+\mu}{ \mu + e^{-\tau z_c}} - 1\right] \nonumber \\
&=  \frac{(\gamma z_c+2)( \mu + e^{-\tau z_c})  - 2-2\mu}{ \mu + e^{-\tau z_c}} .
\end{align}
We now define $h(z_c):=(\gamma z_c+2)( \mu + e^{-\tau z_c})  - 2-2\mu$ and proceed to show that $h(z_c)\geq0$ holds for $0\geq z_c \geq - \delta$.  This is established from the following sequence of implications:
\begin{subequations}
\begin{align} 
h(z_c)=&(\gamma z_c+2)( \mu + e^{-\tau z_c})  - 2-2\mu \\    \label{eq::gzpositive1}
= & (\gamma z_c+2)\left(\mu + \sum _{n=0}^{\infty }{\frac{(-\tau z_c)^{n}} {n!}}  \right) - 2 - 2\mu \\\label{eq::gzpositive2}
\geq & (\gamma z_c+2)\left(\mu + 1- \tau z_c + \frac{(\tau z_c)^2}{2} \right) - 2 - 2\mu \\\label{eq::gzpositive3}
=& \gamma z_c \left( \mu+1  -\frac{2\tau}{\gamma} -\tau z_c + \frac{\tau^2 z_c}{\gamma} +\frac{\tau^2 z_c^2}{2}\right )  \\\label{eq::gzpositive4}
=& \gamma \tau z_c^2 \left( \frac{\mu+1}{2} -1 + \frac{\tau z_c}{2} \right ) \\\label{eq::gzpositive5}
\geq & \gamma \tau z_c^2 \left( \frac{\mu-1-\log(2+\mu)}{2} \right)\\ \label{eq::gzpositive6}
\geq & 0.
\end{align}
\end{subequations}
Here, \eqref{eq::gzpositive2} follows because $\gamma z_c+1 >0$ and $-\tau z_c \geq 0$. \eqref{eq::gzpositive4} follows since $\tau:=\frac{\mu+1}{2}\gamma$. 
\eqref{eq::gzpositive5} follows since $z_c \geq - \delta$. \eqref{eq::gzpositive6} follows because $\mu-1-\log(2+\mu)$ is a monotonically increasing function for $\mu\geq 0$ and $\mu\geq \bar{\mu}$. For  $z_c \geq 0 $ we have,
\begin{align} 
h'(z_c) &= \gamma \mu + (\gamma-\tau\gamma z_c-2\tau) e^{-\tau z_c}\nonumber \\
       &=  \gamma \mu + 2(\gamma-\tau)e^{-\tau z_c} - \gamma\left (\frac{\tau z_c+1}{e^{\tau z_c}}\right)\nonumber \\
       &\geq \gamma \mu + 2(\gamma-\tau) -\gamma\nonumber\\   \label{eq::g'zpositive3}
       &=0.
\end{align}
This follows because $\gamma-\tau < 0$ ($ \bar{\mu}>2$), $0<e^{-\tau z_c}\leq1$, and $\frac{\tau z+1}{e^{\tau z_c}}\leq1$. Since $h(0) =0$ we have that  $h(z_c)\geq0$ for $z_c \geq 0 $. We thus have that $\text{SigVaR}_{1-\alpha}^{\mu,\tau}(f(x_c,\Xi)) \leq 0$ holds for $\mu,\tau$ satisfying the conditions of the proposition.  $\Box$
\end{proof}

Proposition \ref{thm:svarvscvar} is of practical computational relevance because it indicates that we can use the solution of CVaR-P (which is a computationally attractive formulation) to find an initial guess for SigVar-P.  We also note that Proposition \ref{thm:svarvscvar} implies that SigVaR provides an approximation that is at least as good as that of EVaR. 

\subsection{Relationship with DC-P}
The following results compare the solutions of SigVaR-P and DC-P. To establish these results, we define the SigVaR-DC error:
\begin{align} 
d_{\mu,\tau} &:=     \mathbb{E}[\psi_{ss}^{\mu,\tau}(Z)]-\epsilon^{-1}\mathbb{E}\left[ \left[Z+\epsilon\right]_{+} - \left[Z\right]_{+} \right] . 
\end{align}
In addition, we define $d_{\mu,\tau}(z) :=     \psi_{ss}^{\mu,\tau}(z)- \epsilon^{-1}\left[ \left[z+\epsilon\right]_{+} - \left[z\right]_{+} \right]$ for all $z\in \mathbb{R}$. Consequently, $d_{\mu,\tau} =\mathbb{E}[d_{\mu,\tau}(Z)]$. \\

We now establish a lower bound for the SigVar-DC error. 
\begin{proposition}
Assume that $\tau\in\mathbb{R}_+$ satisfies $\tau \leq  \frac{1}{2}\epsilon^{-1}$. We have that $d_{\mu,\tau} \geq 0$ for any $\mu\in \mathbb{R}_+$. 
\label{thm:svarvsDC2}
\end{proposition}
\begin{proof}  
We proceed by proving that $d_{\mu,\tau}(z) \geq 0$ holds for any $z\in\mathbb{R}$. If $z < - \epsilon$ we have that $ \epsilon^{-1}\left[ \left[z+\epsilon\right]_{+} - \left[z\right]_{+} \right]   = 0$  and, consequently,  $d_{\mu,\tau} \geq 0$.  For $z \geq  - \epsilon$ we have that,
\begin{align} 
 2\frac{1+\mu}{ \mu + e^{-\tau z}}-1 \geq&  2\frac{1+\mu}{ \mu + e^{\tau\epsilon}}-1 \geq0.
\end{align}
We also observe that, for $ - \epsilon \leq z\leq 0$, 
\begin{align} 
d_{\mu,\tau} &=  \left[2\frac{1+\mu}{ \mu + e^{-\tau z}} - 1\right] - \left[\epsilon^{-1}z+1\right]   =  \frac{-\hat{h}(z)}{ \mu + e^{-\tau z}} .
\end{align}
We proceed to show that $ \hat{h}(z):=(\epsilon^{-1} z+2)( \mu + e^{-\tau z})  - 2-2\mu \leq 0$ holds for  $ - \epsilon \leq z\leq 0$.  This is established from the following sequence of implications:
\begin{align} 
\hat{h}'(z) &=  \epsilon^{-1}\mu + \left( \epsilon^{-1}- \epsilon^{-1}\tau z-2\tau\right) e^{-\tau z} \nonumber\\
       &\geq  \epsilon^{-1}\mu + \left ( \epsilon^{-1}-2\tau\right) e^{-\tau z}\nonumber\\ 
       &\geq  \epsilon^{-1}\mu. 
\end{align}
Here, the first inequality follows since $ z\leq 0$ and the second inequality follows because of the condition $\tau \leq  \frac{1}{2}\epsilon^{-1}$. 
Since $\hat{h}(0) =0$,  we have that $\hat{h}(z) \leq 0$ for$ - \epsilon \leq z\leq 0$. For  $z \geq 0 $ we have, $d_{\mu,\tau} =   \left[2\frac{1+\mu}{ \mu + e^{-\tau z}} - 1\right] -1 \geq  0$. $\Box$
\end{proof}

This result shows that as $\epsilon\to 0$, the range of feasible $\tau$ that make SigVaR-P more conservative increases. We now establish an upper bound for the SigVar-DC error. 
\begin{proposition}
Assume $\mu,\tau\in\mathbb{R}_+$ satisfy $\mu\geq \bar{\mu}$ where $\bar{\mu}$ is the positive solution of ${\bar{\mu}-\log(2+\bar{\mu})}=1$ and $ \tau \geq \frac{1}{2}\epsilon^{-1} (\mu+1)$. We have that $d_{\mu,\tau} \leq \frac{2}{\mu}$. 
\label{thm:svarvsDC}
\end{proposition}
\begin{proof}   
We proceed by proving that $d_{\mu,\tau} \leq \frac{2}{\mu}$ holds for any $z\in\mathbb{R}$ if $\mu,\tau$ satisfy the conditions of the proposition. If $z < - \delta$, where $\delta := \frac{1}{\tau} \log(2+\mu)$, we have that $2\frac{1+\mu}{ \mu + e^{-\tau z}} - 1<0$  and, consequently,  $d_{\mu,\tau}\leq0$.  For  $- \epsilon \leq z < 0 $, we can follow the derivation of Proposition \ref{thm:svarvscvar} to prove that $d_{\mu,\tau}\leq0$. For  $z \geq 0 $ we have that $d_{\mu,\tau} =    \left[2\frac{1+\mu}{ \mu + e^{-\tau z}} - 1\right] - 1 \leq \frac{2}{\mu}$. The result follows. $\Box$

\end{proof}
This result shows that improving the quality of the DC-P  approximation (by setting $\epsilon\to 0$) corresponds to setting $\mu,\tau\to \infty$ for SigVaR-P (e.g., by using $\tau(\mu)=\theta(\mu+1)$ with $\theta=\frac{1}{2}\epsilon^{-1}$).  

\subsection{Relationship with SS-P}
The following results compare the solutions of SigVaR-P and SS-P. We show that there exist parameters of SigVaR-P that provide an approximation of CC-P that is at least as good as that of SS-P.

\begin{proposition}
Assume that $\mu,\tau\in\mathbb{R}_+$ satisfy $\tau = \frac{1}{\rho}$, $\mu =   \frac{2+\rho m_1}{\rho m_2}$. We have that $\mathcal{X}^{\rho,m_1, m_2}_{sm}(\alpha)\subseteq \mathcal{X}_{ss}^{\mu,\tau}(\alpha)$ and $\varphi_{ss}^{\mu,\tau_\alpha}(\alpha) \leq \varphi^{\rho,m_1,m_2}_{sm}(\alpha)$. 
\label{thm:svarvssmooth}
\end{proposition}
\begin{proof}
We proceed by proving that any feasible point $x\in \mathcal{X}^{\rho,m_1, m_2}_{sm}(\alpha)$ of SS-P is a feasible point for SigVaR-P provided that $\mu,\tau$ satisfy the conditions of the proposition. This would imply that we can always find $\mu,\tau$ such that $\mathcal{X}^{\rho,m_1, m_2}_{sm}(\alpha)\subseteq \mathcal{X}_{ss}^{\mu,\tau}(\alpha)$ and $\varphi_{ss}^{\mu,\tau_\alpha}(\alpha) \leq \varphi^{\rho,m_1,m_2}_{sm}(\alpha)$. It suffices to show that $\frac{1+\rho m_1}{1+\rho m_2 e^{-z/\rho }} \geq [2\frac{1+\mu}{ \mu + e^{-\tau z}} - 1]_{+}$ holds for any $z\in\mathbb{R}$. If $z < - \delta$, where $\delta := \frac{1}{\tau} \log(2+\mu)$, we have that $2\frac{1+\mu}{ \mu + e^{-\tau z}} - 1<0$  and, consequently,  $\frac{1+\rho m_1}{1+\rho m_2 e^{-z/\rho }} \geq [2\frac{1+\mu}{ \mu + e^{-\tau z}} - 1]_{+}$.  For $z \geq - \delta$ we have
\begin{subequations}
\begin{align} 
& \frac{1+\rho m_1}{1+\rho m_2 e^{-z/\rho }} - \left[2\frac{1+\mu}{ \mu + e^{-\tau z}} - 1\right]_{+} \\
=& \frac{1+\rho m_1}{1+\rho m_2 e^{-z/\rho }} - 2\frac{1+\mu}{ \mu + e^{-\tau z}} + 1  \\
=&  \frac{1+\rho m_1}{1+\rho m_2 e^{-z/\rho }} - \frac{4+2\rho m_1+2\rho m_2 }{ 2+\rho m_1  + \rho m_2 e^{-z/\rho}} + 1  \\
=&  \frac{ (1+\rho m_1)^2 - (1+\rho m_2)^2  + (\rho m_2)^2 (e^{-z/\rho }-1)^2   }{(1+\rho m_2 e^{-z/\rho}) \cdot (2+\rho m_1  + \rho m_2 e^{-z/\rho })}    \\
\geq& 0
\end{align}
\end{subequations}
where the first equality holds since $2\frac{1+\mu}{ \mu + e^{-\tau z}} - 1\geq 0$, the second equality follows by substituting  $\tau = \frac{1}{\rho}$ and $\mu =   \frac{2+\rho m_1}{\rho m_2}$, and the inequality holds since $m_2 \leq m_1$. $\Box$
\end{proof}

\begin{corollary}
Assume that $\mu,\tau\in\mathbb{R}_+$ satisfy $\tau = (1+\mu)\theta$, $\theta=\frac{m_2}{2+\rho m_1 + \rho m_2}$, and $\mu \geq \frac{2+\rho m_1}{\rho m_2}$, we have that $\mathcal{X}^{\rho,m_1, m_2}_{sm}(\alpha)\subseteq \mathcal{X}_{ss}^{\mu,\tau}(\alpha)$ and $\varphi_{ss}^{\mu,\tau}(\alpha) \leq \varphi^{\rho,m_1,m_2}_{sm}(\alpha)$. 
\label{col:svarvssmooth}
\end{corollary}

\section{Computational Implementation} \label{sec:comp}

We use SAA to convert SigVar-P into a finite-dimensional NLP \cite{kleywegt2002sample}. We generate a set of realizations $\xi\in \Omega$ from $p_\Xi$. The total number of realizations is $S$. The SAA approximation is given by:
\begin{subequations}\label{eq:saa}
\begin{align}
\min \limits_{x \in  \mathcal{X}, z_\xi\in \mathbb{R}^S,\phi_\xi\in \mathbb{R}^S_+} \;\;  &\varphi(x)  \\
{\rm s.t.}  \;\;     & z_\xi = f(x,\xi),  \;\;  \xi\in\Omega  \\
					  & \phi_\xi \geq 2\frac{1+\mu}{ \mu + e^{- \tau z_\xi}}-1, \;\;  \xi\in\Omega \\
					  &{\frac{1}{|\Omega|}}\sum_{\xi\in\Omega}  \phi_\xi  \leq \alpha. 
\end{align}
\end{subequations}

Large values of $\tau$ will cause difficulty for the NLP solver due to the high nonlinearity of the sigmoid function. For example, the first derivative of $2\frac{1+\mu}{ \mu + e^{- \tau z_\xi}}$ with respect to $z_\xi$ is $\mathcal{O}(\tau)$ and thus becomes increasingly steep as $\tau$ is increased. Moreover, the second derivative is $\mathcal{O} (\tau^2)$.  Consequently, we propose a scheme to solve a sequence of SigVaR approximations of increasing quality and with this achieve more robustness. The scheme (called {\tt SigVaR-Alg}) begins by finding a solution of the SAA approximation of the CVaR-P.  The SAA approximation of CVaR-P is:
\begin{subequations}\label{eq:saacvar}
\begin{align}
\min \limits_{x \in  \mathcal{X}, z_\xi\in \mathbb{R}^S,\phi_\xi\in \mathbb{R}^S_+,t\in\mathbb{R}} \;\;  &\varphi(x)  \\
{\rm s.t.}  \;\;     & z_\xi = f(x,\xi),  \;\;  \xi\in\Omega  \\
					  & \phi_\xi \geq z_\xi-t, \;\;  \xi\in\Omega \\
					  &{\frac{1}{|\Omega|}} \sum_{\xi\in\Omega}  \phi_\xi  \leq {-t} \alpha. 
\end{align}
\end{subequations}

\begin{algorithm}[!htb] \caption{\tt SigVaR-Alg}
\begin{algorithmic}[]
\item[\bf{1.} Initialize]
\STATE Given $\lambda >1$, $\alpha\in (0,1]$, and target $\mu^*\in \mathbb{R}_+$. 
\STATE Initialize iteration index $\ell \leftarrow 0$.
\STATE Solve CVaR problem \eqref{eq:saacvar} and set $\gamma\leftarrow-\frac{1}{t_c(\alpha)}$, $x_\ell^*\leftarrow x_c(\alpha)$, and $\varphi_\ell^* \leftarrow \varphi_c(\alpha)$.
\STATE Set $\mu_\ell\leftarrow\bar{\mu}$,  $\tau_\ell\leftarrow\frac{\mu_\ell+1}{2}\gamma$, where $\bar{\mu}$ is positive solution of ${\bar{\mu}-\log(2+\bar{\mu})}=1$.
\STATE Update iteration index $\ell \leftarrow \ell+1$.
\item[\bf{2.} Solve SigVar-P]
\STATE Use $x_{\ell-1}^*$ as initial guess and solve SigVaR-P \eqref{eq:saa} with $\mu_\ell,\tau_\ell$. 
\STATE Set  $x_{\ell}^*\leftarrow x_{ss}^{\mu_\ell,\tau_\ell}(\alpha)$ and $\varphi_{\ell}^*\leftarrow \varphi_{ss}^{\mu_\ell,\tau_\ell}(\alpha)$.
\IF{$\mu_\ell\geq \mu^*$}
      \STATE   Go to Step {\bf 4}.
\ELSE 
	 \STATE   Go to Step {\bf 3}.     
\ENDIF
\item[\bf{3.} Update parameters]
\STATE{ Set $\mu_{\ell+1} \leftarrow \lambda \cdot \mu_{\ell}$ and $\tau_{\ell+1}\leftarrow\frac{\mu_{\ell+1}+1}{2}\gamma$. }
\STATE Update iteration index $\ell \leftarrow \ell+1$ and return to Step {\bf 2}.
\item[\bf{4.} Stop with $x_{\ell}^*$]
\end{algorithmic}
\end{algorithm}

From Proposition \ref{thm:svarvscvar}, we have that $\varphi_1^* \leq \varphi_{0}^*$ holds  and from Lemma \ref{thm:monotonical} we have that $\varphi_{\ell+1}^* \leq \varphi_{\ell}^* $ holds for all $\ell\geq 1$ (provided that the NLPs are solved to global optimality). However, for the numerical studies in Section \ref{sec:results}, the SigVaR approximation at each iteration is solved to local optimality because solving a large-scale NLP to global optimality is computationally intractable.

\section{Numerical Studies}\label{sec:results}
The first two case studies are small-scale linear problem; consequently, exact and tractable MILP reformulations can be used and provide best performance. We use two small-scale studies to illustrate the theoretical properties of SigVaR. The next two case studies include a wind turbine optimization study and a flare system optimization study, which are large-scale and highly nonlinear. For these two case studies, exact mixed integer reformulations are intractable. We use the large-scale studies to illustrate the practical benefits of SigVaR.


\subsection{Analytical Example}
Consider the following CC-P:
\begin{subequations}
\label{example_eqns}
\begin{align}
\min \limits_{x\in \mathbb{R} } \;\;  &x  \label{example_eqns1} \\
{\rm s.t.}  \;\; &\mathbb{P}( \Xi \leq x) \geq 1- \alpha,  \label{example_eqns2}
\end{align}
\end{subequations}
with $\Xi \sim \mathcal{U}(0,1)$. The optimal objective value and solution are $\varphi(\alpha)=x^*(\alpha)=1-\alpha$ and we note that $\mathbb{P}( \Xi \leq x^*(\alpha))=1-\alpha$. This implies $1-\alpha=F(x^*(\alpha))=Q_{1-\alpha}(\Xi)=x^*(\alpha)$. We handle the CC \eqref{example_eqns2} using the VaR (exact), the CVaR approximation \eqref{cvar_formulation}, the EVaR approximation \eqref{evar_formulation}, and the SigVaR approximation \eqref{svar_formulation}.The optimal solution and objective values obtained with these approaches are, respectively,  $\text{VaR}_{1-\alpha}(\Xi)=Q_{1-\alpha}(\Xi)$, $\text{CVaR}_{1-\alpha}(\Xi)$, $\text{EVaR}_{1-\alpha}(\Xi)$, and  $\text{SigVaR}_{1-\alpha}^{\mu,\tau}(\Xi)$. Moreover, $\text{VaR}_{1-\alpha}(\Xi) = 1- \alpha$, $\text{CVaR}_{1-\alpha}(\Xi)=\frac{1}{2}(2-\alpha)$, and $\text{EVaR}_{1-\alpha}(\Xi)=\inf \limits_{t > 0} \{t \log(te^{t^{-1}} -t) -t \log \alpha  \}$.   For the case of SigVar we have that, for $\alpha \geq \frac{2+2\mu}{\mu\tau} \log(\frac{2+\mu+\mu e^\tau}{2+2\mu}) - 1$, 
\begin{equation} 
\text{SigVaR}_{1-\alpha}^{\mu,\tau}(\Xi) = \tau^{-1} \log\left(\frac{\mu e^\tau-\mu\beta}{\beta-1}\right) 
\end{equation}
where $\beta=e^{\frac{(\alpha+1]\mu\tau}{2+2\mu}}$. Otherwise, we have that
\begin{equation}
\text{SigVaR}_{1-\alpha}^{\mu,\tau}(\Xi) {=}\inf_{t\in\mathbb{R}}\left\{\frac{2+\mu}{\mu\tau} \log(2+\mu) + \frac{2+2\mu}{\mu\tau} \log\left(\frac{\mu e^{\tau(1-t)} + 1}{2+2\mu}\right) + t - 1 \leq \alpha \right\}.
\end{equation}

The optimal objective values for all approaches as a function of $\alpha$ are shown in Figure \ref{fig.example}.  As predicted by the properties of SigVaR, we have that $\text{VaR}_{1-\alpha}(\Xi) \leq \text{SigVaR}_{1-\alpha}^{\mu,\tau}(\Xi)$ for all $\alpha$.

\begin{figure}[tbhp]
\begin{center}
\includegraphics[width=0.6\textwidth]{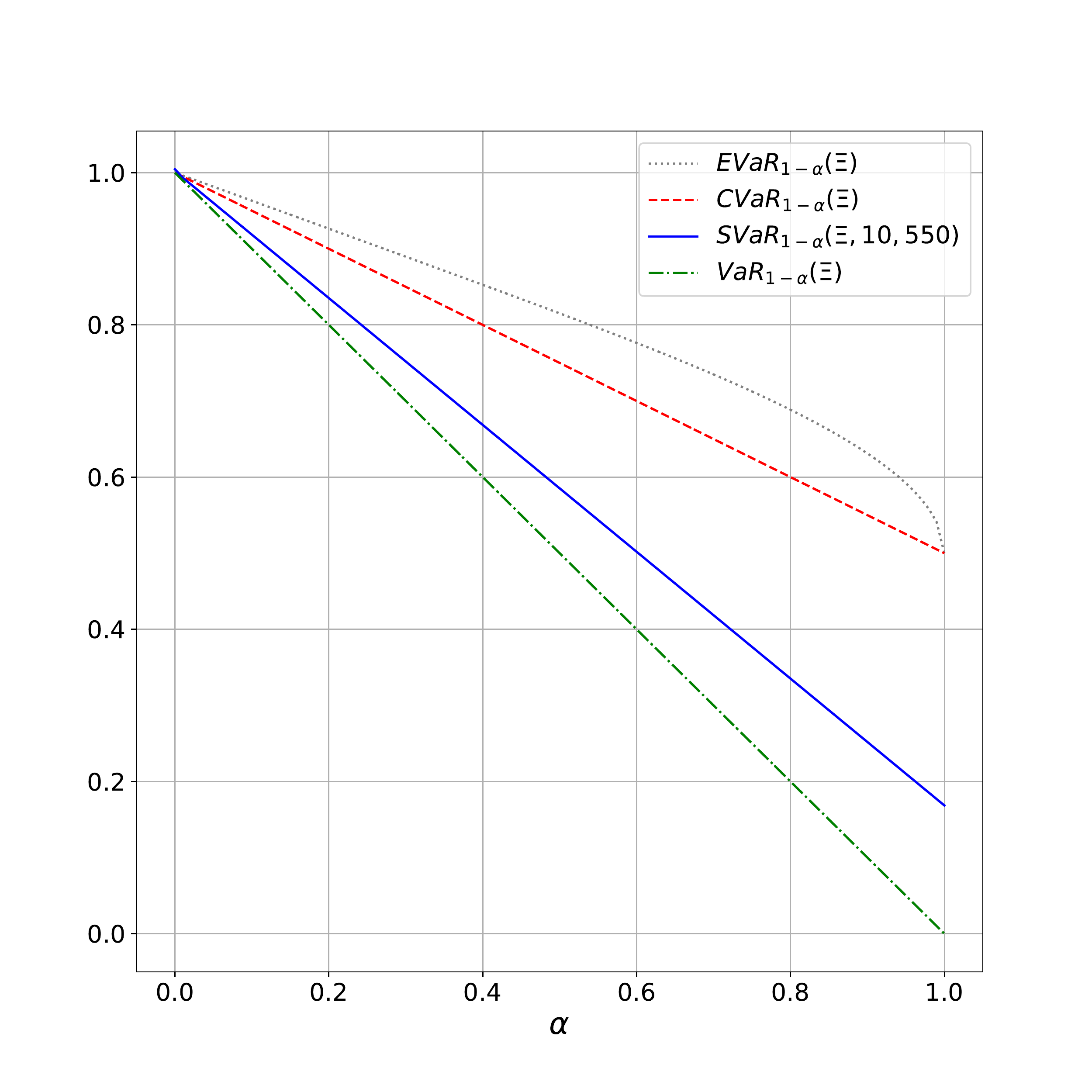}\vspace{-0.2in}
\caption{Optimal objectives obtained with VaR, CVaR, EVaR, and SigVaR for analytical example.}
\label{fig.example}
\end{center}
\end{figure}

We have that $Z(x)=\Xi-x\sim \mathcal{U}(-x,1-x)$ for $x\in\mathbb{X}$. Consequently, the constant $L=\sup_{x\in \mathcal{X}}L(x) = \sup_z \{p_{Z(x)}(z)\}=1$ satisfies $\mathbb{P}(-\delta\leq Z(x) < 0) \leq L\delta$ for all $x\in \mathcal{X}$. From Lemma \ref{bounderror}, the approximation error of the SigVaR function is bounded as $\epsilon_{\mu,\tau} \leq\frac{\log(2+\mu)L}{\tau} +  \frac{2}{\mu} =  \frac{\log(12)}{550} +  \frac{2}{550} = 0.204$. We note that this is an upper bound of the empirical error $\epsilon^{\mu,\tau}=0.169$ observed in Figure \ref{fig.example} and computed by $\epsilon_{\mu,\tau}=\text{SVaR}_{1-\alpha}^{\mu,\tau}(\Xi)-\text{VaR}_{1-\alpha}(\Xi)$ (vertical distance at each $x^*(\alpha)=1-\alpha$).  

From the solution of the CVaR approximation we obtain that $t_c(\alpha) =-\frac{\alpha}{2}<0$ and thus $\gamma_\alpha=-\frac{1}{t_c(\alpha)}=\frac{2}{\alpha}$. Proposition \ref{thm:svarvscvar} predicts that for $\mu=10,\tau=550$ and $\alpha = 0.02$,  $\text{SigVaR}_{1-\alpha}^{\mu,\tau}(\Xi) \leq \text{CVaR}_{1-\alpha}(\Xi)$. This prediction is verified in Figure \ref{fig.example}, which shows that empirically for $\alpha>0.006$, $\text{SigVaR}_{1-\alpha}^{\mu,\tau}(\Xi) \leq \text{CVaR}_{1-\alpha}(\Xi)$. The extreme conservatism of CVaR and EVaR becomes obvious at large values of $\alpha$.  In particular, at $\alpha=1$ we see that $\text{SigVaR}_{1-\alpha}^{\mu,\tau}(\Xi)=0.169$ and $\text{CVaR}_{1-\alpha}(\Xi)$=0.5, which illustrates that the quality of the approximation can be substantially improved. 

In practice, it is very rare that analytical solutions can be obtained. Therefore, We now illustrate numerical behavior of {\tt SigVar-Alg} (in our experiments we use SAA with 1,000 scenarios). The CC-P in this case can be cast exactly as an MILP, CVaR-P is cast as an LP, and SigVaR-P and as NLP. The MILPs are solved with the solver  {\tt SCIP} and the LPs and NLPs are solved with {\tt IPOPT}. Lemma \ref{thm:monotonical}  shows that $\text{SigVaR}_{1-\alpha}^{\mu,\tau(\mu)}(\Xi)$ becomes less conservative for increasing $\mu$, which is verified in Figure \ref{fig:examplealpha} for $\alpha=0.5$ and $\alpha=0.05$.  For $\alpha=0.5$, the solution of the MILP formulation is 0.504, which is close to the analytical solution of 0.5. {\tt SigVaR-Alg} first finds the solution of CVaR approximation, which is 0.747. At iteration 1, we solve with SigVaR approximation with $\mu=2.5$ and $\tau=7.2$, and find a solution of 0.719. After 8 iterations, we solve a SigVaR approximation with $\mu=321$ and $\tau=662$ and find a a solution of 0.515. The gap between MILP formulation and SigVaR is only 4\% of the gap between CC-P and CVaR-P.  For $\alpha=0.1$, the gap is 36\% but we also see that the gap is more difficult to close with SigVaR.


\begin{figure}[tbhp]
\centering
\subfloat{\includegraphics[width=0.5\textwidth]{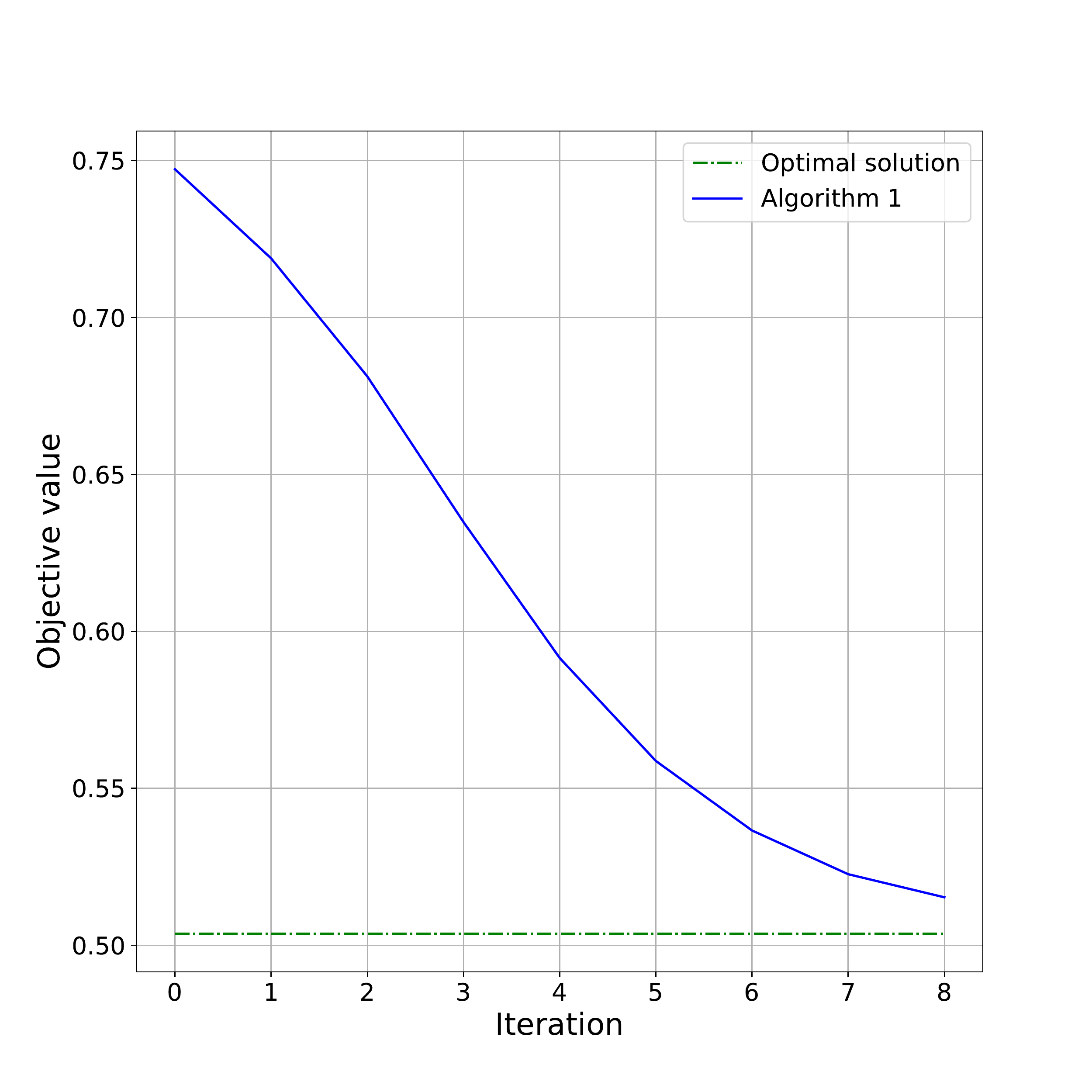}\label{fig:example0.5}}
\subfloat{\includegraphics[width=0.5\textwidth]{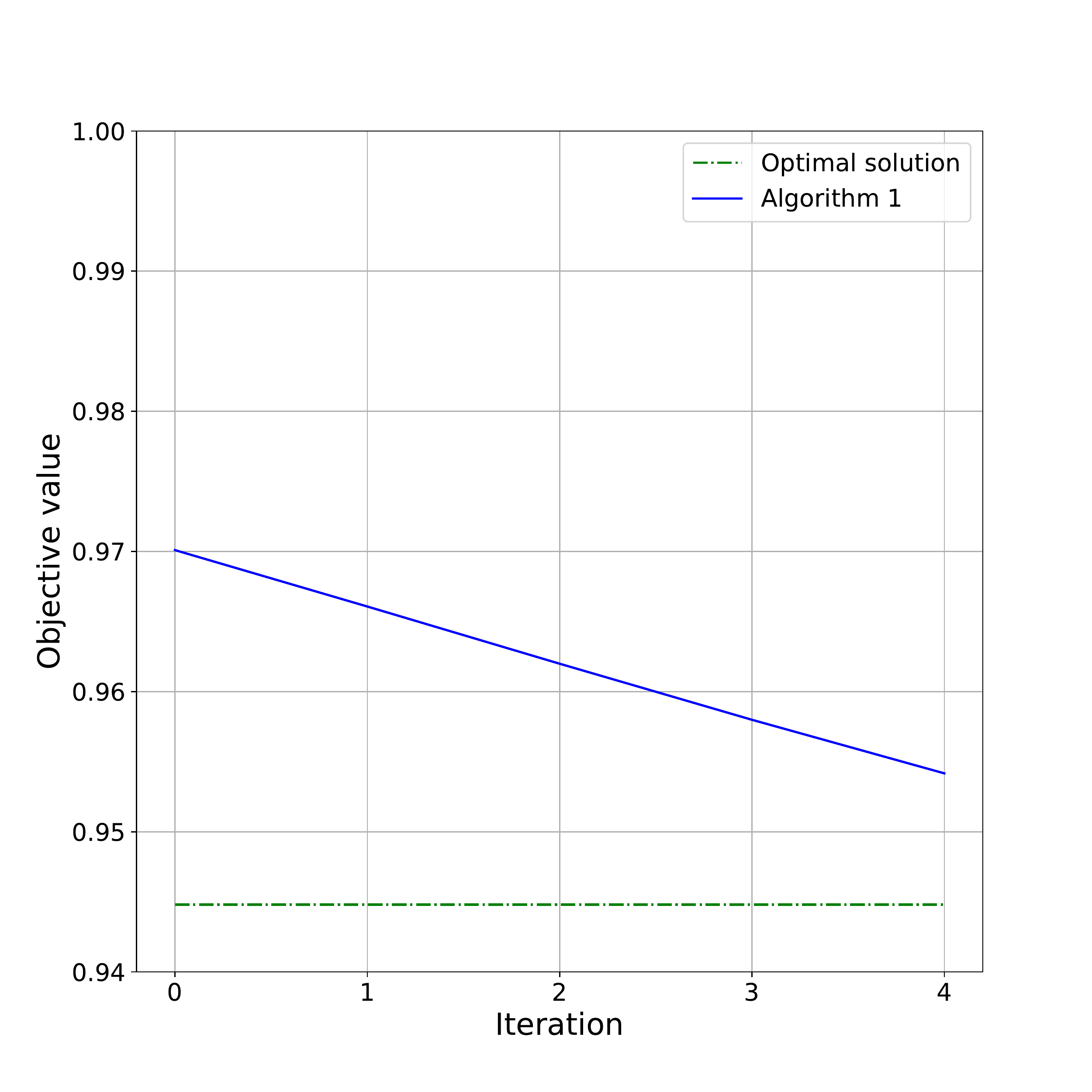}\label{fig.example0.05}}
\vspace{-0.2in}
\caption{Performance of SigVar on analytical example with $\alpha = 0.5$ (left) and $\alpha = 0.05$ (right).}
\label{fig:examplealpha}
\end{figure}

\subsection{Farmer Problem}
We consider modified version of the classical farmer problem \cite{birge2011introduction}. In this problem, the farmer needs to decide how much land to allocate to grow wheat, corn, and beets while considering the uncertainty on crop yields. The farmer has the option to buy/sell crops to satisfy contracts and maximize revenue (minimize cost). The formulation is given by:
\begin{subequations}\label{eq:Farmer}
\begin{align}
\min_{x,y_j(\cdot),w_j(\cdot)} &\quad \varphi=\mathbb{E}\left[ f(\Xi) \right]\\
\textrm{s.t.} &\sum_{j\in\mathcal{P}}x_j\leq \bar{x}\\
& \tau_j(\Xi) x_j + y_j(\Xi) -w_j(\Xi)\geq \beta_j,j\in\mathcal{P}\quad\textrm{a.s.} \\
& f(\Xi) = \sum_{j\in\mathcal{P}   }\left(\gamma^x_j x_j+\gamma^y_j y_j(\Xi) - \gamma^w_jw_j(\Xi)\right)\quad \textrm{a.s.}  \\
&\mathbb{P}\left(f(\Xi) \leq \bar{f} \right) \geq 1- \alpha \label{Farmer_eqns1}\\
& 0\leq w_j(\Xi)\leq \bar{w}_j,\; 0\leq y_j(\Xi)\leq \bar{y}_j,j\in\mathcal{P}\quad\textrm{a.s.}
\end{align}
\end{subequations}
where $x_j$ denotes the land allocated to each crop at cost $\gamma_j^x$, $y_j(\xi)$ represents the crops bought at price $\gamma_j^y$,  $w_j(\xi)$ denotes the crops sold at price $\gamma_j^w$, $\mathcal{P}$ denotes the set of crops $\{\textrm{wheat,corn,beets}\}$,  $\tau_j(\xi)$ is the yield of crops, $\beta_j$ denotes demand contracts and $\bar{x},\bar{y}_\ell,\bar{w}_\ell$ represents capacities. Constraint \eqref{Farmer_eqns1} requires that the cost $f(\cdot)$ is lower than the threshold $\bar{f}$ with probability at least $1-\alpha$. We assume that  the yield of wheat and corn is constant, while the yield of beets follows a normal distribution $\mathcal{N}(20,5)$. We generate 1,000 scenarios from this distribution and we set $\alpha = 0.05$ and $\bar{f}= \$ 50,000$.  

The performance of {\tt SigVaR-Alg} is summarized in Table \ref{tb:farmer}. The solution of CC-P is obtained using the MILP formulation. As can be seen, the expected cost of the MILP formulation is \$-86431. The expected cost of CVaR approximation is \$-76455 (which is around 11.5\% higher than the optimal MILP cost). This is because, although \eqref{Farmer_eqns1} only requires the cost to be lower than the threshold with probability equal to larger than 0.95, the solution of CVaR formulation satisfies the constraint with probability 0.978. Figure \ref{fig:farm} shows the histogram of the cost obtained with CVaR, SigVaR, and MILP formulations. Here, it becomes obvious that CVaR can significantly distort the cost distribution due to high conservatism. From the solution of the CVaR approximation we obtain  $t_c(\alpha)=-5555<0$ and $\gamma=0.00018>0$. After 6 iterations, {\tt SigVaR-Alg} solves the SigVaR approximation with $\mu=80$ and $\tau=0.0073$ and finds a solution with an expected cost of \$-85472 (which is is around 1.1\% higher than the optimal MILP cost). The gap between the MILP and SigVaR formulations is only 9.6\% of the gap between the MILP and CVaR formulations. We also observe that, as the iterations proceed, the objective value of SigVaR-P decreases monotonically, $\mathbb{P}(f(\Xi)\leq  \bar{f})$ decreases, and $\text{VaR}_{1-\alpha}(f(\Xi))$ increases.  We can thus see that the SigVaR formulation can significantly reduce the conservatism of the CVaR solution. We acknowledge, however, that we are unable to close the gap further due to numerical instability of the NLP solver.  

\begin{table}[tbhp]
\caption{Performance of {\tt SigVaR-Alg} on farmer problem with $\alpha = 0.05$.}
\label{tb:farmer}
\centering
\begin{footnotesize}
\begin{tabular}{|c|c|c|c|c|c|}
\hline
$\ell$ 
&$\mu$
&$\tau$
& $\mathbb{E}[f(\Xi)]$
& $\text{VaR}_{1-\alpha}(f(\Xi))$
& $\mathbb{P}\left(f(x,\Xi)\leq  \bar{f} \right)$ \\
\hline
CVaR-P($\ell=0$) &  - & - & -76455      &-55601        &0.978  \\
1	   & 2.5    &0.00031      & -78396      &-54511        &0.974        \\
2      &5.0     &0.00054      &-80225       &-53484      & 0.969          \\
3      &10.0    &0.00098      &-82141       &-52408       &0.965    \\
4      &20.0    &0.00188      &-83659       &-51556       &0.959    \\
5      &40.0    &0.00367      &-84746        &-50945      &0.957       \\
6  	   &80.0    &0.00725       &-85472       &50538        &0.953          \\
CC-P    &   -  &  -   & -86431       & -50000  & 0.95  \\
\hline		
\end{tabular}
\label{tab:farm}
\end{footnotesize}
\end{table}

\begin{figure}[tbhp]
\centering
\subfloat{\includegraphics[width=0.33\textwidth]{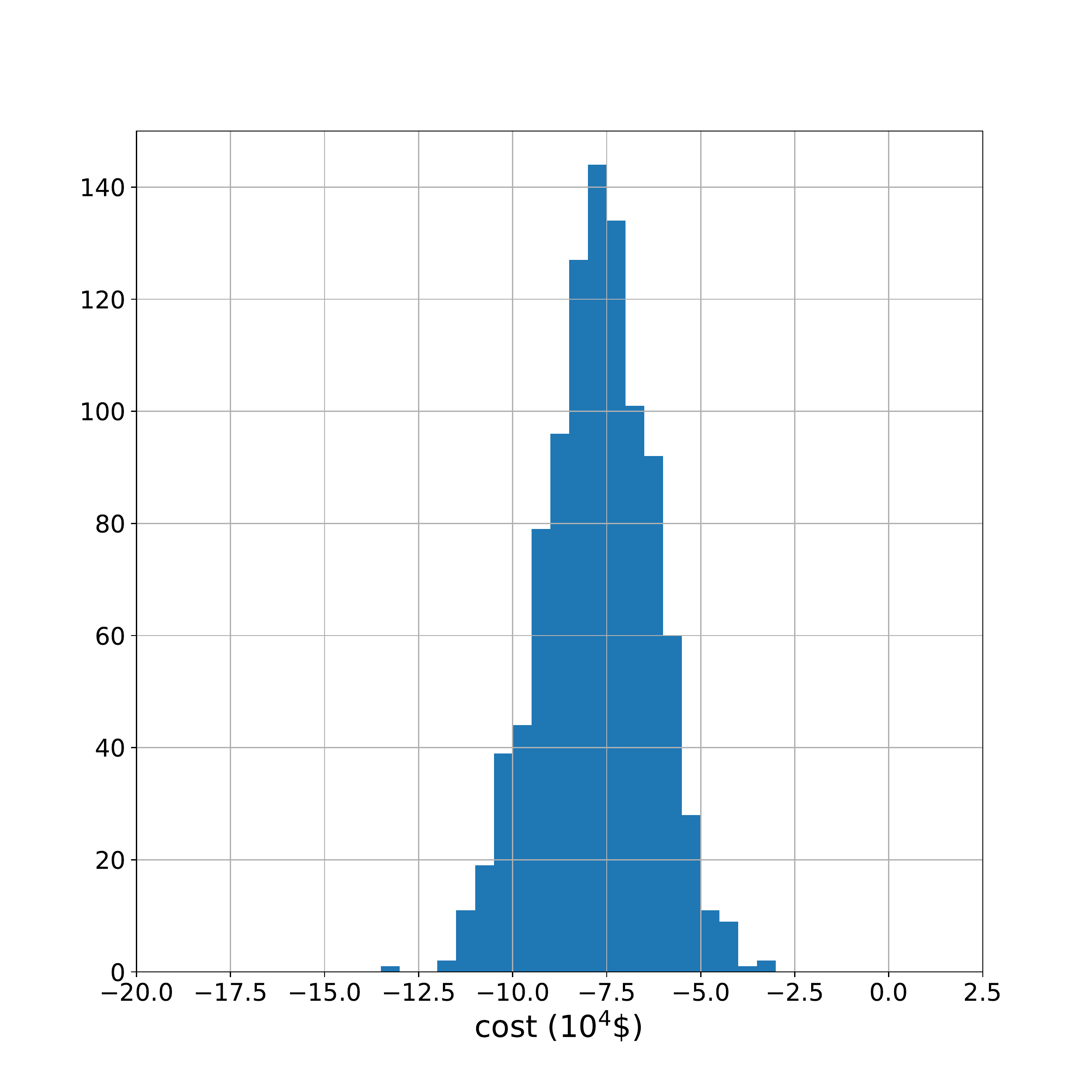}\label{fig:FarmCVaR}}
\subfloat{\includegraphics[width=0.33\textwidth]{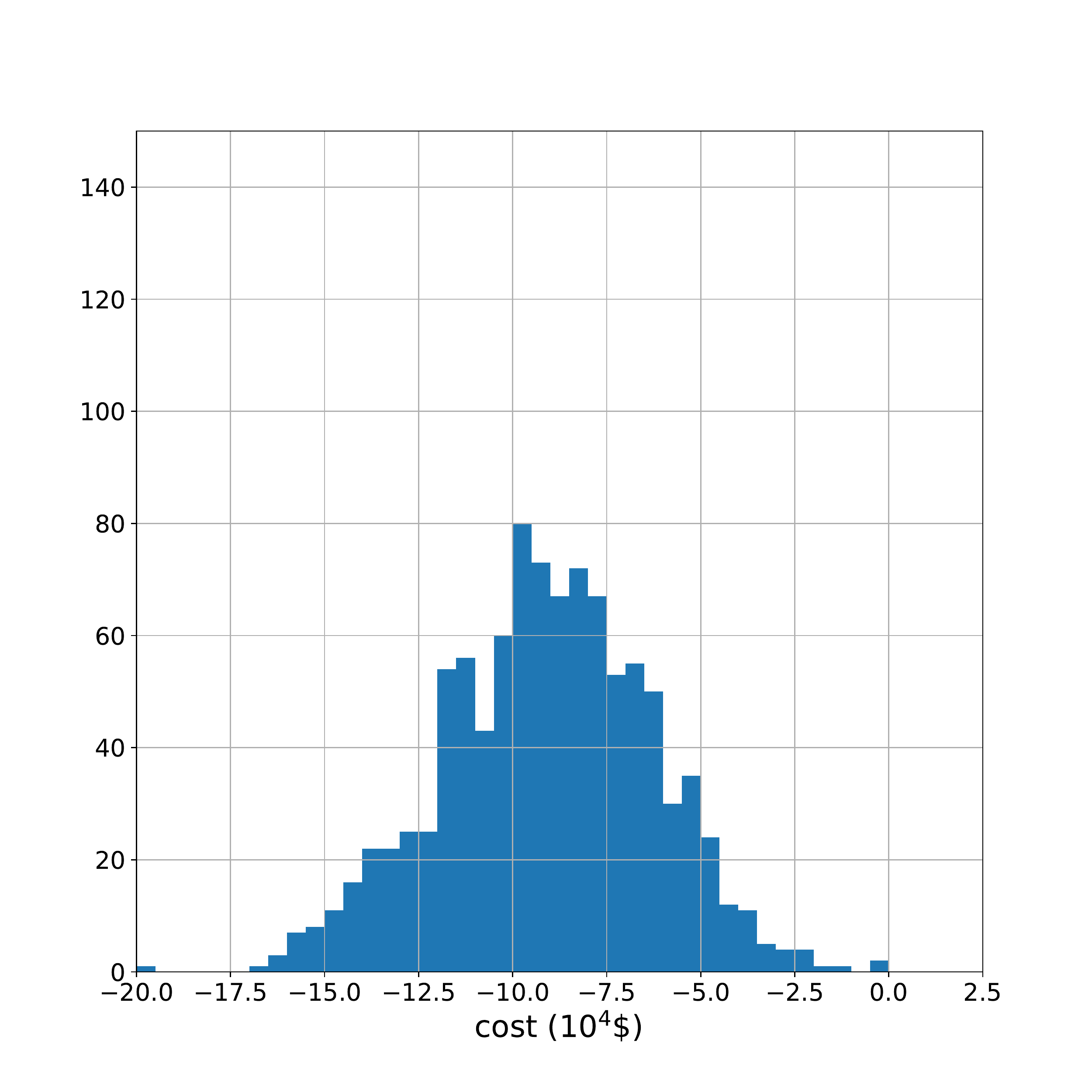}\label{fig:FarmSigVaR}}
\subfloat{\includegraphics[width=0.33\textwidth]{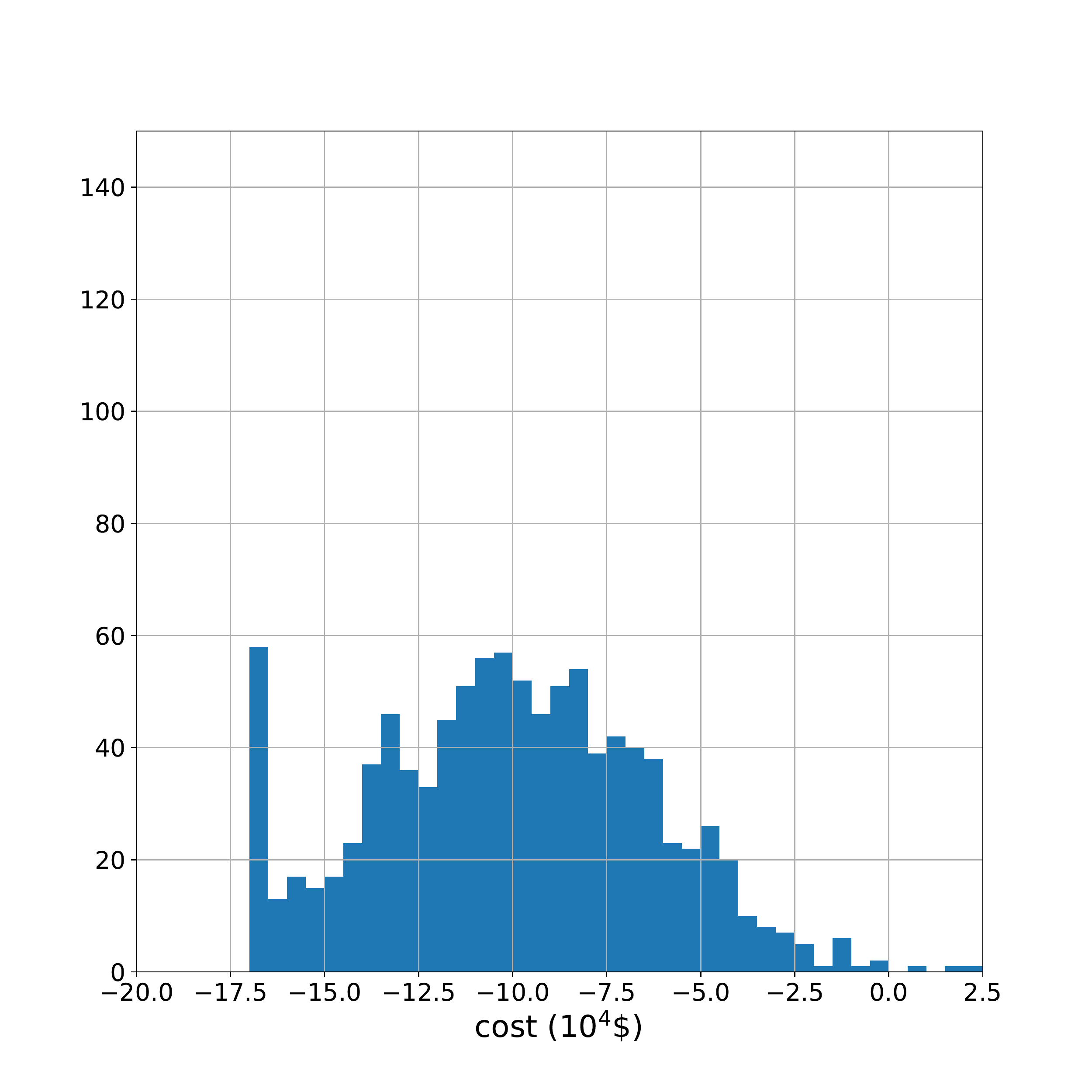}\label{fig:FarmVaR}}
\vspace{-0.1in}
\caption{Cost distribution using CVaR-P (left), SigVaR-P (middle) and CC-P (right) formulation.}	
\label{fig:farm}
\end{figure}

\subsection{Wind Turbine Optimization}
We now solve a large-scale CC-P that seeks to find optimal pitch and torque control policies for a wind turbine given uncertainty in wind speed conditions. The formulation seeks to maximize expected power and to satisfy a CC on the maximum mechanical load experienced by the wind turbine. We represent this problem in the following abstract form:
\begin{subequations}
\begin{align} 
\max_{u\in \mathcal{U}} &\; \;\varphi:=\mathbb{E}\left[\frac{1}{T}\int_{\mathcal{T}}y_P(t,\Xi)dt\right]   \\
\textrm{s.t.} &\; (y_P(\Xi, t),y_L(\Xi, t))=\mathcal{M}(u(t),u(t,\Xi),V(\Xi, t)),\; t \in \mathcal{T},\; \textrm{a.s.} \label{eq:turbinemodel}  \\
&\;\;  \mathbb{P}\left\{y_L^{max}(\Xi)\leq \bar{y}_L\right\}\geq 1- \alpha \label{eq:probcons1} \\
&\;\; y_L^{max}(\Xi)=\max_{t\in\mathcal{T}}\, y_L(t,\Xi),\; \textrm{a.s.}\\ 
&\;\; y_L(\Xi, t)\leq \hat{y}_L,\; t \in \mathcal{T},\; \textrm{a.s.} \label{eq:probcons2} 
\end{align}
\end{subequations}
where $t\in \mathcal{T}:=[0,T]$,  $V(\Xi, t)$ is the wind speed, $y_P(\Xi, t)$ is the wind turbine power, $y_L(\Xi, t)$ is the mechanical load with associated threshold $\bar{y}_L$. For a time horizon of ten minutes, we set the control actions for the first 10 seconds to be first stage variables $u(t)$ (the implemented control actions) and the rest to be second stage variables $u(t,\Xi)$ (the recourse control actions). Equation \eqref{eq:turbinemodel} is an abstract representation of a wind turbine model (which comprises nonlinear differential and algebraic equations). The model details are presented in \cite{windpaper}.  A {\tt Julia} model implementation along with all necessary data is available at \url{https://github.com/zavalab/JuliaBox/tree/master/WindSigVaR}. 

An important practical problem is that power maximization conflicts with the mechanical load experienced by the turbine (i.e., the higher the power extracted the higher the load). Consequently, it is  important to carefully trade-off these metrics so as to prevent putting the turbine at extreme mechanical risk. The probabilistic constraint \eqref{eq:probcons1} enforces that the probability that the peak load $y_L^{max}(\Xi)$ exceeds the threshold $\bar{y}_L$ is no more than $\alpha$. Constraint \eqref{eq:probcons2} enforces that the peak load never exceeds another (less conservative) threshold $\hat{y}_L$. In our experiments we set $\alpha=0.5, \bar{y}_L= 60$ MNm,  and $\hat{y}_L = 200$ MNm. 

To solve this problem, we discretize the dynamic model by using a Radau collocation scheme \cite{zavalathesis}. To accurately capture extreme loads we have found that it is necessary to discretize the model using a resolution of 0.5 seconds over 10 minutes, giving rise to 1,200 time steps. For an NLP with 230 scenarios (collected from real implementations), the total number of variables  is 5.5 million. The NLPs arising in this application were implemented in  {\tt Plasmo.jl} \cite{plasmo} and solved with the parallel interior-point solver {\tt PIPS-NLP} \cite{pipsnlp} (which exploits the structure of the stochastic program at the linear algebra level) and with the off-the-shelf serial solver {\tt IPOPT}\cite{ipopt} (which treats the problem as a general NLP). Because of the size of the problem and because the wind turbine model is nonconvex, MINLP  formulations of CC-P  are computationally intractable. A conservative approach to solve this problem is to enforce the load constraint for all scenarios (almost surely).  The expected power using this approach is 3.5 MW.

Table \ref{tb:Wind} summarizes the performance of {\tt SigVaR-Alg}. The serial solver {\tt Ipopt} takes 0.8 hours to solve the CVaR-P while the parallel solver {\tt PIPS-NLP} requires 30 minutes using 23 computing cores. The expected power obtained with CVaR-P  is 3.548 MW and we have found this performance to be too conservative. In particular, although the CC \eqref{eq:probcons1} only requires $\max_{t\in\mathcal{T}}\, y_L(t,\Xi)\leq \bar{y}_L$ to hold with a probability of 0.5, the CVaR-P solution satisfies it with probability 0.748. From the solution CVaR-P we obtain $\gamma_\alpha=0.822$.  From Table \ref{tb:Wind} we also see that the SigVaR approximation becomes less conservative as we increase $\mu,\tau$ and that the objective value is progressively improved (power is maximized). After three iterations, {\tt SigVaR-Alg} solves SigVaR-P with $\mu=10$ and $\tau=4.52$ and achieves an expected power of 3.865 MW (an improvement of 8.9\% over CVaR-P). The probability of satisfying the maximum load threshold is reduced to 0.583. At a price of electricity of 30 \$/MWh, these cost savings obtained with SigVaR-P translate to around \$83,000 per year (for a single 5 MW wind turbine).  We can thus see that the economic benefits of reducing conservatism can be quite significant. 

\begin{table}[tbhp]
\caption{Performance of {\tt SigVaR-Alg} on wind turbine optimization problem with $\alpha = 0.5$.}
\label{tb:Wind}
\centering
\begin{footnotesize}
\begin{tabular}{|c|c|c|c|c|c|c|c|}
\hline
$\ell$
&$\mu$ 
&$\tau$
& $\varphi$   
& $\text{VaR}_{1-\alpha}(y_L^{max}(\Xi))$ 
& $\mathbb{P}\left\{y_L^{max}(\Xi)\leq \bar{y}_L\right\}$ 
& Time 
&{\tt Ipopt} \\
&&  &$(MW)$ & $(MNm)$ & &(Hour) & Iter\\
\hline
CVaR-P      &  -         & -              & 3.548         &47.85       &0.748  &0.8 &160  \\
1     &2.5          &1.44           &3.766       &49.56   &0.726  &2.9  &603  \\
2     &5.0          &2.47            &3.835       &52.12  &0.643  &1.2 &238\\ 
3     &10.0        &4.52            &3.865       &54.52  &0.583 &1.3 &256\\
\hline		
\end{tabular}
\end{footnotesize}
\end{table}

Figure \ref{fig:Load} shows the cost distribution for the maximum load obtained with the CVaR-P and SigVaR-P. It is clear that CVaR is significantly more conservative and pushes the mechanical load towards small values. SigVaR, on the other hand, allows for an equal proportion of load violations and with this it can extract more power. This is illustrated in Figure \ref{fig:Power}, where we show that SigVaR achieves a larger proportion of scenarios with a large power output.

\begin{figure}[!htp]
\centering
\subfloat{\includegraphics[width=0.5\textwidth]{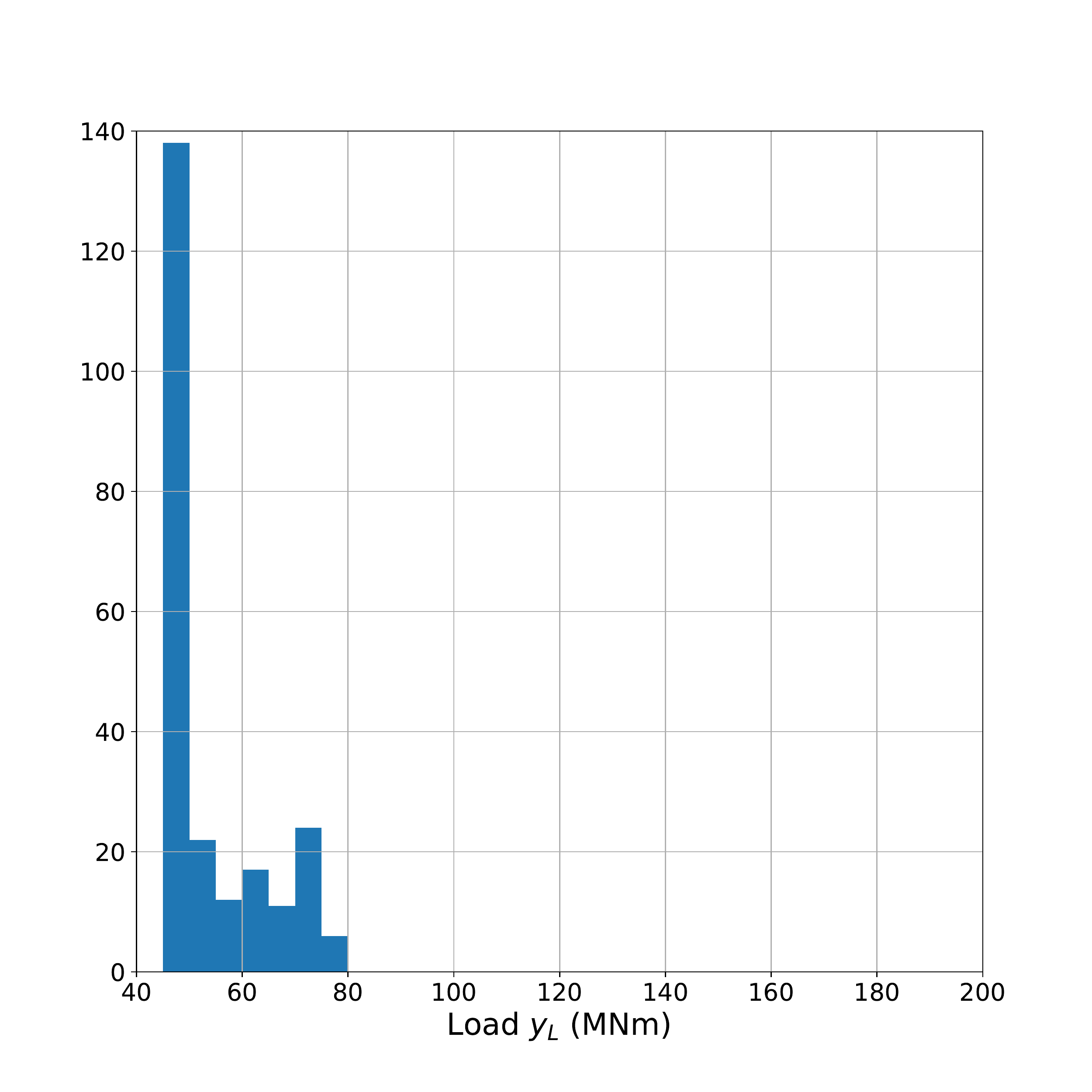}\label{fig:CVaRLoad}}
\subfloat{\includegraphics[width=0.5\textwidth]{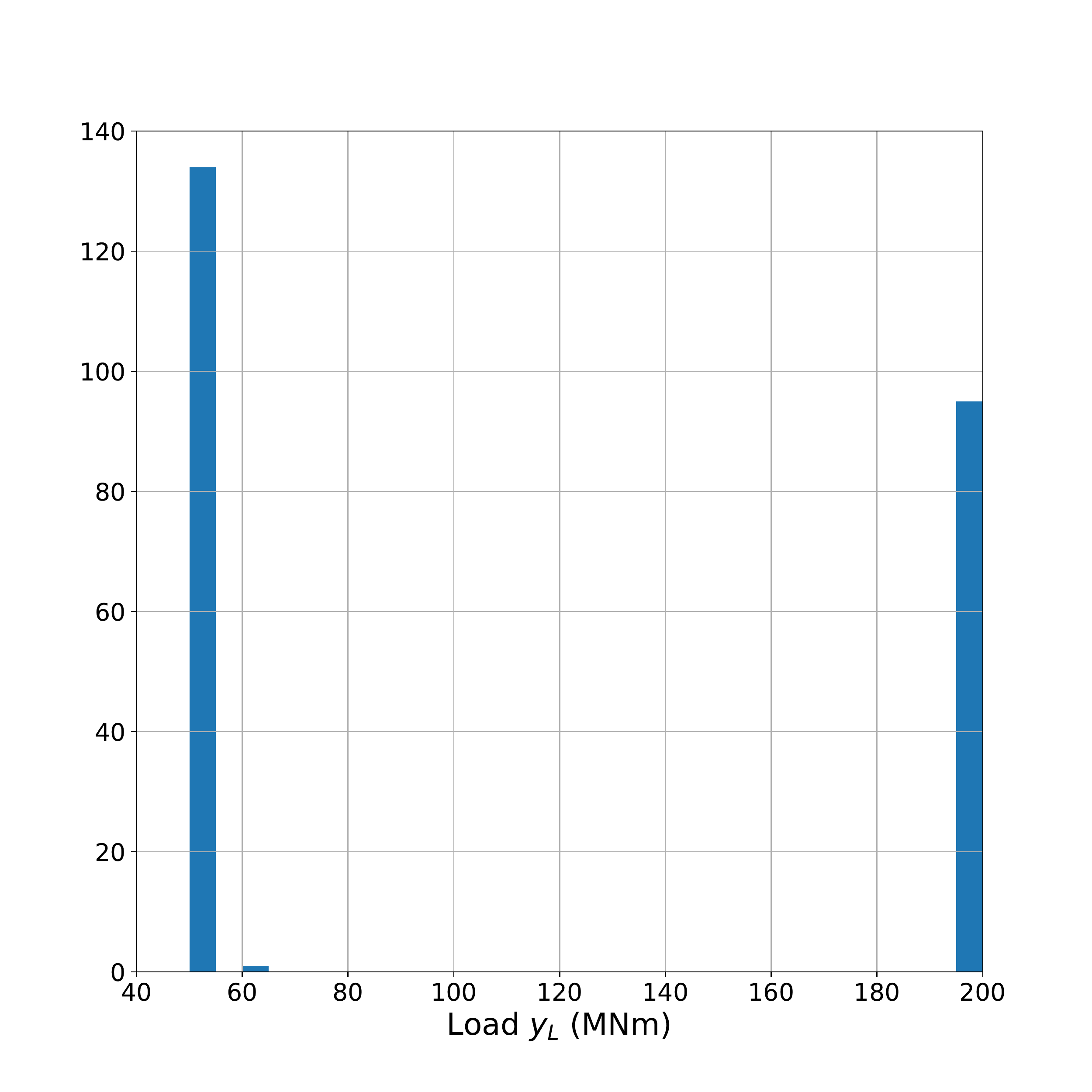}\label{fig:SigVaRLoad}}
\caption{Histogram of mechanical load using CVaR (left) and SigVaR (right) formulation.}	
\label{fig:Load}
\end{figure}

\begin{figure}[!htp]
\centering
\subfloat{\includegraphics[width=0.5\textwidth]{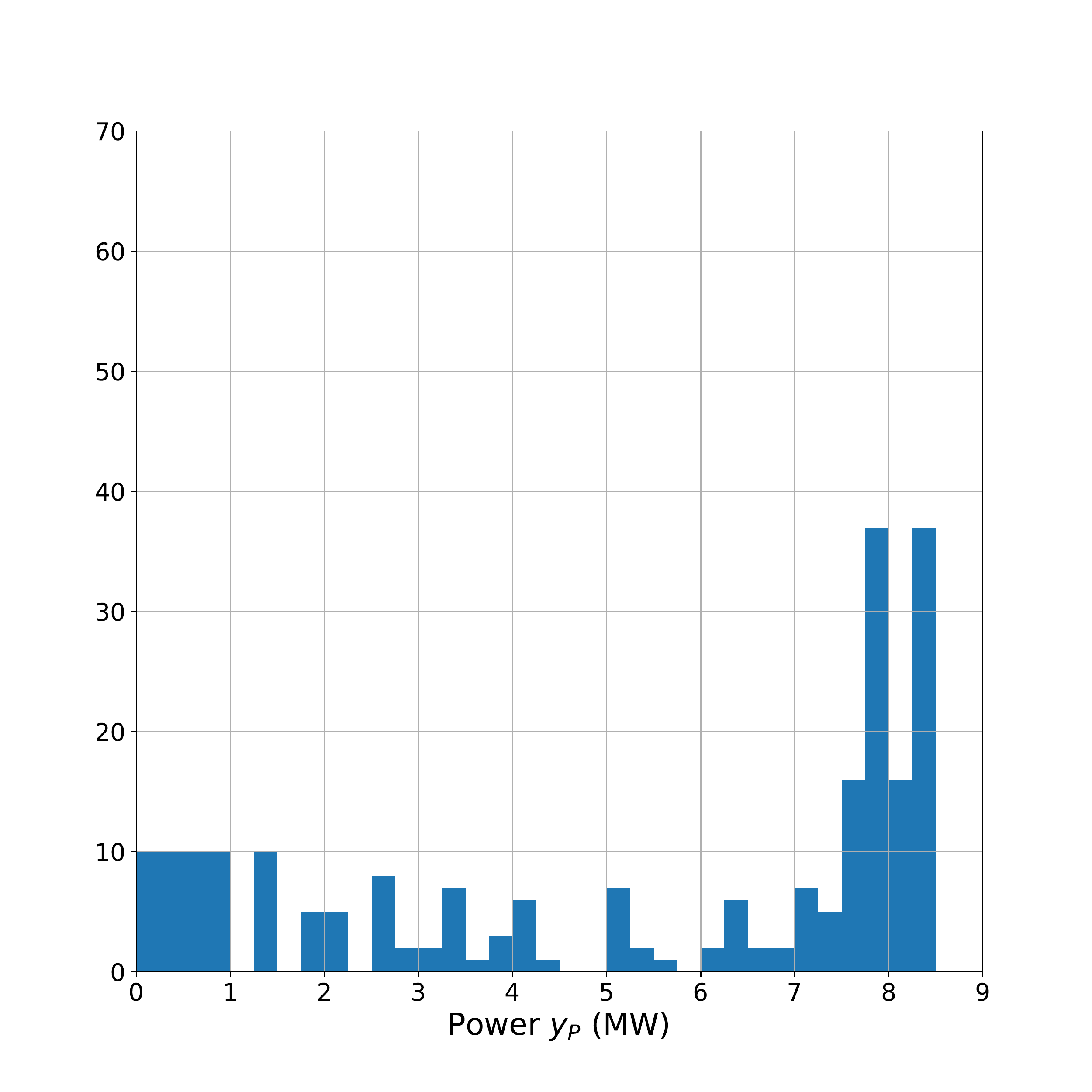}\label{fig:CVaRPower}}
\subfloat{\includegraphics[width=0.5\textwidth]{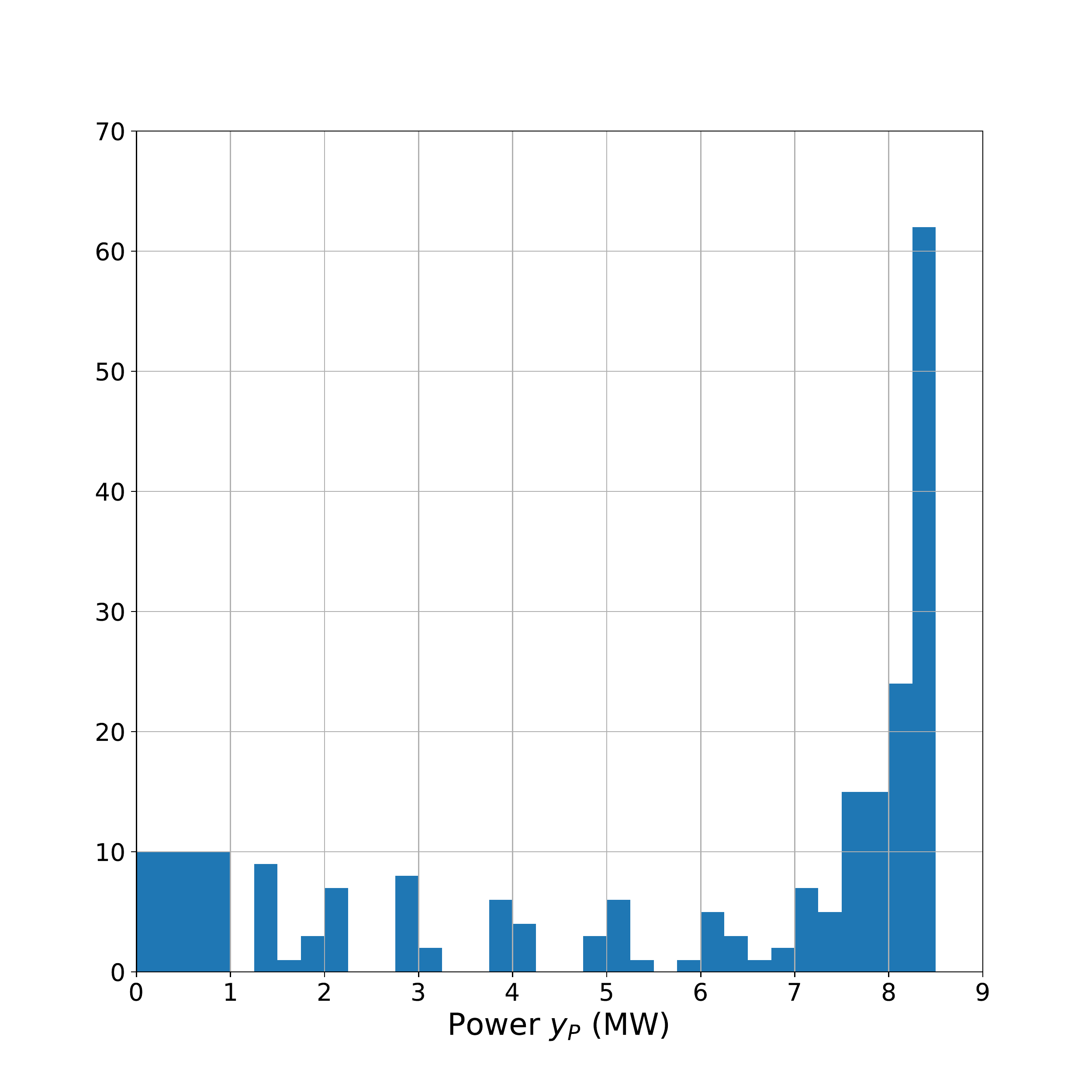}\label{fig:SVaRPower}}
\caption{Histogram of wind turbine power using CVaR (left) and SigVaR (right) formulation.}	
\label{fig:Power}
\end{figure}

Table \ref{tb:Wind_smooth} summarizes the performance of smooth sigmoidal approximation SS-P. Here we set $m_1=1$ and $m_2=0.5$, the same as the Figure 1 in  \cite{geletu2015tractable}. We tried 10 different values of $\rho$. When $\rho \geq 6.25$, the approximation is too conservative and there is no feasible solution. When $\rho=0.195$, {\tt IPOPT} has numerical difficulty in solving the problem. The best expected power obtained with smooth approximation is better than the solution from CVaR-P, but worse than the solution from {\tt SigVaR-Alg}. These results highlight the importance of having an explicit connection between SigVaR-P and CVaR-P and with this obtain an  initial guess for the parameter values. 

When $\alpha=0.05$, the expected power obtained with CVaR-P  is 3.5 MW, which is the same as the expected power obtained by forcing the inequality constraint to hold for all scenarios. Both SigVaR-P and SS-P cannot further improve the performance.

\begin{table}[tbhp]
\caption{Performance of SS-P on wind turbine optimization problem with $\alpha = 0.5$.}
\label{tb:Wind_smooth}
\centering
\begin{footnotesize}
\begin{tabular}{|c|c|c|c|c|c|c|c|}
\hline
$\ell$
&$\rho$ 
& $\varphi$   
& $\text{VaR}_{1-\alpha}(y_L^{max}(\Xi))$ 
& $\mathbb{P}\left\{y_L^{max}(\Xi)\leq \bar{y}_L\right\}$ 
& Time 
&{\tt Ipopt}\\
& &$(MW)$ & $(MNm)$ &  &(hr) & Iter\\
\hline
1     &100                      &  -     &   -      &  -      &   -          &  -      \\
2     &50                        &  -     &  -      &  -      &   -          &  -    \\ 
3     &25                        &  -     &  -      &  -      &   -          &  -    \\ 
4     &12.5                     &  -     &  -      &  -     &    -         &  -     \\
5     &6.25                     &  -     &  -      &  -     &    -         &  -    \\
6     &3.125                    &1.538       &5.53  & 1.0   &2.1 &461\\
7     &1.563                  &2.909       &30.79  & 1.0   &2.5 &553\\
8     &0.781                &3.427      &34.02  & 0.778 &5.2 &1138\\
9     &0.390               &3.749       &42.55  & 0.669 &2.2   &464\\
10     &0.195                &-         &-  & -  &-  & - \\
\hline		
\end{tabular}
\end{footnotesize}
\end{table}

\subsection{Flare System Optimization}

We consider the design of a flare stack system that combusts a random waste fuel gas flow.  Gas flares are used as safety (relief) devices to manage abnormal situations in infrastructure systems (natural gas and oil processing plants and pipelines), manufacturing facilities (chemical plants, offshore rigs), and power generation facilities. Abnormal situations include equipment failures, off-specification products, and excess materials in start-up/shutdown procedures. In particular, flares prevent over-pressuring of equipment and use combustion to convert flammable, toxic or corrosive vapors to less dangerous compounds \cite{epa2017}. Flare design is influenced by several uncertain factors such as the amount and composition of the waste flow stream to be combusted and the ambient conditions. These systems are currently designed based on typical historical values for waste fuel gases and ambient conditions \cite{api1997521,epa2017}. Consequently,  an improperly designed flare can be susceptible to extreme events not experienced before.  The design goals are to minimize capital cost while controlling the radiation level at ground level (which is a function of the input waste flow to be combusted). 

The heat released by combustion $H$ (BTU/h) is a function of the random input waste flow $Q$ (lb/h) and the heat of combustion $h_c$ (BTU/lb):
\begin{equation} \label{eq:heat}
H =h_c\,Q
\end{equation}
The flame length $L$ (ft) can be calculated as a function of the released heat using an approximation of the form:
\begin{equation} \label{eq:flame}
\log L=a_1\log H-a_2
\end{equation}
The flare stack diameter $t$ (ft) is sized on a velocity basis. This is done by relating this to the Mach number $M$ and the waste flow as:
\begin{equation} \label{eq:match}
M^2=\frac{a_3}{t^2} Q^2.
\end{equation}
The flare tip exit velocity $U$ (ft/s) is function of the flow and the diameter:
\begin{equation} \label{eq:vel}
U=a_4 \,\frac{Q}{t^2}
\end{equation}
The wind speed $w$ (ft/s) is an important environmental factor that affects the tilting of the flame and the distance from the centre of the flame. The following correlations capture the flame distortion as a result of the wind speed and the exit velocity:
\begin{equation} \label{eq:delx}
\log\Delta X=\log(a_5\, L )+a_6\,(\log w-\log U)
\end{equation}
\begin{equation} \label{eq:dely}
\log\Delta Y=\log(a_7\, L )-a_{8}\,(\log w-\log U)
\end{equation}
Here, $\Delta X$ and $\Delta Y$ (ft) are the horizontal and vertical distortions. The distortions are used to compute the horizontal $X$, vertical $Y$, and total distance $D$ (ft) to a given ground-level safe point $(r,0)$ as:
\begin{align}\label{eq:diam}
X&=r-\frac{1}{2}\,\Delta X\\
Y&=h+\frac{1}{2}\,\Delta Y\\
D^2&={X}^2+{Y}^2.
\end{align}
Here, $h$ (ft) is the flare height. The flame radiation $K$ (BTU/h\,ft${}^2)$ is a function of the heat released and the total distance:
\begin{equation}\label{eq:rad}
K=a_9\frac{H }{D^2}.
\end{equation}
A primary safety goal in the flare stack design problem is to control the risk that the radiation exceeds a certain threshold value $\bar{k}$ (BTU/h\,ft${}^2$) at the ground-level reference point $(r,0)$. This is modeled using the CC:
\begin{align}\label{eq:cck}
\mathbb{P}(K\leq \bar{k})\geq 1-\alpha.
\end{align}
The objective function is the cost (USD), which is a function of height and diameter:
\begin{equation}\label{eq:cost}
\varphi(t,h)=(a_{10}+a_{11}\,t+a_{12}\,h)^2.
\end{equation}
The height and the diameter are key design parameters that control the radiation experienced at the reference point (i.e., a higher and wider flare reduces the radiation intensity). As a result, there is an inherent trade-off between capital cost and safety that needs to be carefully handled. The overall goal of the optimization problem is to determine the optimal value of the height and diameter. We assume that the random input waste flow follows an exponential distribution (with a rate parameter  21, 000 lb/h). We generate 2000 scenarios from this distribution and we set $\alpha=0.05$. The total number of variables in the NLPs is on the order of 18,000.  A  {\tt Julia} model implementation along with all necessary data and parameters (e.g. $a_1-a_{12}$) is available at \url{https://github.com/zavalab/JuliaBox/tree/master/FlareDesignSigVaR}. 

A conservative solution is first obtained by enforcing radiation constraint for all scenarios. The cost associated with this approach is \$ 149,284. Table \ref{tb:Flare} summarizes the performance of {\tt SigVaR-Alg}. The cost obtained with CVaR-P  is \$ 121,170. Although this approximation has reduced the cost by 18.8\% compared with the scenario approach, this performance is still too conservative. In particular, although the CC only requires $K(\Xi)\leq \bar{k}$ to hold with a probability of 0.95, the CVaR-P solution satisfies it with probability 0.979. From the solution CVaR-P we obtain $\gamma_\alpha=0.0026$.  From Table \ref{tb:Flare} we also see that the SigVaR approximation becomes less conservative as we increase $\mu,\tau$ and that the objective value is progressively improved. After eight iterations, {\tt SigVaR-Alg} solves SigVaR-P with $\mu=320$ and $\tau=0.42$ and reduced the cost to \$ 109,488, which is of 9.6\% lower than the cost of CVaR-P. The probability of satisfying chance constraint is reduced to 0.951. We can thus see that the economic benefits of reducing conservatism can be quite significant. 

Table \ref{tb:Wind_smooth} summarizes the performance of the smooth SS-P approximation. For the first 4 iterations, the cost does not monotonically decrease as we increase the value of $\rho$. This might be due to the fact that there are multiple local optimal solutions. When $\rho \leq 0.39$, {\tt IPOPT} has numerical difficulty in solving the problem. The cost obtained with this  approximation is 3.8\% lower than the solution obtained with CVaR-P, but 6.4\% higher than the cost of {\tt SigVaR-Alg}.

\begin{table}[tbhp]
\caption{Performance of {\tt SigVaR-Alg} on flare system optimization with $\alpha = 0.05$.}
\label{tb:Flare}
\centering
\begin{footnotesize}
\begin{tabular}{|c|c|c|c|c|c|c|c|}
\hline
$\ell$
&$\mu$ 
&$\tau$
&Cost  
& $\text{VaR}_{1-\alpha}(K(\Xi))$ 
& $\mathbb{P}\left\{K(\Xi)\leq \bar{k}\right\}$ 
& Time 
&Ipopt \\
& &  &(USD) &(Btu/(hr $\text{ft}^2$)) &   &(sec) & Iter\\
\hline
CVaR-P&- & -  &121,170     &1612      &0.979   &20      &143   \\
1     &2.5         &0.0045   &118,176     &1687   &0.975   &3  &37  \\
2     &5.0         &0.0079   &115,893    &1767   &0.971    &4  &52\\
3     &10.0        &0.0144   &113,815   &1833  &0.965    &4 &49\\
4     &20.0	      &0.0275    &112,258  &1885      &0.962   &6      &65   \\
5     &40.1		 &0.0537	 &111,134   &1923      &0.957   &9      &97   \\
6     &80.2		 &0.1061    &110,328   &1952      &0.954   &15      &121   \\
7     &160		 &0.211    &109,780    &1971      &0.951   &104     &512   \\
8     &320		 &0.420    &109,488   &1982      &0.951    &45      &439   \\
\hline		
\end{tabular}
\end{footnotesize}
\end{table}

\begin{table}[tbhp]
\caption{Performance of SS-P on flare system optimization with $\alpha = 0.05$.}
\label{tb:Flare_smooth}
\centering
\begin{footnotesize}
\begin{tabular}{|c|c|c|c|c|c|c|}
\hline
$\ell$
&$\rho$ 
&Cost  
& $\text{VaR}_{1-\alpha}(K(\Xi))$ 
& $\mathbb{P}\left\{K(\Xi)\leq \bar{k}\right\}$ 
& Time 
&Ipopt \\
&  &(USD) &(Btu/(hr $\text{ft}^2$)) &   &(sec) & Iter\\
\hline
1     &100              &138,865         &1214      &0.997    &119       &708  \\
2     &50           	   &141,880         &1161      &0.998    &58       &445\\
3     &25          	   &142,004         &1159      &0.998    &106       &698\\
4     &12.5	            &135,526         &1277      &0.996    &150       &951  \\
5     &6.25 		 	    &131,540         &1359      &0.993    &80       &723 \\
6     &3.125		    &126,018         &1486      &0.988    &85       &565\\
7     &1.563 		    &122,023         &1589      &0.981    &9      &78\\
8     &0.781		    &116,472         &1749      &0.972    &5       &35\\
9     &0.390		    &-         &-      &-    &-       &- \\
10   &0.195	        &-         &-      &-    &-       &- \\
\hline		
\end{tabular}
\end{footnotesize}
\end{table}

\section{Concluding Remarks}\label{sec:final}
We have proposed a sigmoidal approximation for chance constraints that we call SigVaR. We prove that SigVaR is conservative and that the level of conservatism can be made arbitrarily small for limiting values of the approximation parameters. We also provide conditions for the parameters guaranteeing that the SigVaR approximation is less conservative than the conditional value at risk (CVaR) approximation and other smooth sigmoidal approximations available in the literature. The SigVar approximation brings computational benefits over mixed-integer reformulations because its sample average approximation can be formulated as a standard nonlinear program. We also conduct numerical experiments to demonstrate that it can significantly reduce the conservatism of CVaR. A limitation of SigVaR, however, is that numerical instability is encountered for limiting parameter values. To ameliorate this issue, we proposed an algorithmic scheme that solves a sequence of approximations of increasing quality. This scheme exploits connections between the parameter values of SigVaR and the VaR detected with the CVaR approximation. As part of future work, we are interested in studying more closely the behavior of the sigmoidal approximation from numerical stand-point. In particular, while the proposed scheme does improve numerical performance, extreme sensitivity of the sigmoidal function for large parameter values remains an issue.  

\bibliographystyle{spmpsci}      
\bibliography{chance.bib}
\end{document}